\documentclass[11pt]{amsart}

\usepackage{amsthm}
\usepackage{amssymb}
\usepackage{MnSymbol}
\usepackage{amsxtra}     
\usepackage{epsfig}
\usepackage{verbatim}
\usepackage{url}
\usepackage[all]{xypic}

\theoremstyle{plain}
\newtheorem{thm}{Theorem}
\newtheorem{prop}[thm]{Proposition}

\newtheorem{lem}[thm]{Lemma}

\newtheorem{warning}[thm]{Warning}

\theoremstyle{definition}
\newtheorem{defi}[thm]{Definition}

\theoremstyle{remark}
\newtheorem{rmk}[thm]{Remark}

\newcommand{\adeles}{\mathbb{A}}

\newcommand{\Ind}{\mbox{Ind}}

\newcommand{\hn}{\mathcal{H}_n}

\newcommand{\OK}{\mathcal{O}_{\cmfield}}
\newcommand{\Oe}{\mathcal{O}_E}
\newcommand{\Hom}{\mbox{Hom}}

\newcommand{\isomto}{\overset{\sim}{\rightarrow}}

\newcommand{\ci}{C^{\infty}}

\newcommand{\uo}{\underline{\omega}}

\newcommand{\padic}{p\mathrm{-adic}}

\newcommand{\cpct}{\mathcal{K}}
\newcommand{\cmfield}{K}

\newcommand{\uA}{\underline{A}}

\newcommand{\gln}{\mathrm{GL_n}}
\newcommand{\gl}{\mathrm{GL}}

\newcommand{\gm}{\mathbb{G}_m}

\newcommand{\IR}{\mathbb{R}}
\newcommand{\ZZ}{\mathbb{Z}}
\newcommand{\IC}{\mathbb{C}}
\newcommand{\IQ}{\mathbb{Q}}

\newcommand{\id}{\mbox{Id}}

\newcommand{\End}{\mbox{End}}

\newcommand{\incl}{\mbox{incl}}

\newcommand{\Spec}{\mathrm{Spec}}

\newcommand{\Cp}{\IC_p}
\newcommand{\OCp}{\mathcal{O}_{\Cp}}

\newcommand{\Gal}{\mathrm{Gal}}
\newcommand{\diag}{\mathrm{diag}}

\newcommand{\hern}{\mathrm{Her}_n}

\newcommand{\Sh}{\mathrm{Sh}}
\newcommand{\V}{\mathcal{V}}

\newcommand{\h}[1]{\mathcal{H}_{#1}}
\newcommand{\Auniv}{\mathcal{A}_{\mathrm{univ}}}
\newcommand{\hdr}{H^1_{\mathrm{DR}}}
\newcommand{\ks}{\mathrm{KS}}
\newcommand{\unitroot}{\mathrm{\underline{U}}}

\renewcommand{\paragraph}[1]{\noindent\textbf{#1}. }

\begin{document}

\bibliographystyle{alpha}       
                    
\title[Differential operators, pullbacks, and families]{Differential operators, pullbacks, and families of automorphic forms on unitary groups}

\author{Ellen Eischen}
\thanks{The author's work on this paper was partially supported by National Science Foundation Grants DMS-1249384 and DMS-1559609.}
\address{Ellen Elizabeth Eischen\\
Department of Mathematics\\
University of Oregon\\
Fenton Hall\\
Eugene, OR 97403-1222\\
USA}
\email{eeischen@uoregon.edu}
\maketitle

\begin{abstract}
This paper has two main parts.  First, we construct certain differential operators, which generalize operators studied by G. Shimura.  Then, as an application of some of these differential operators, we construct certain $p$-adic families of automorphic forms.  Building on the author's earlier work, these differential operators map automorphic forms on a unitary group of signature $(n,n)$ to (vector-valued) automorphic forms on the product $U^\varphi\times U^{-\varphi}$ of two unitary groups, where $U^\varphi$ denotes the unitary group associated to a Hermitian form $\varphi$ of arbitrary signature on an $n$-dimensional vector space.  These differential operators have both a $p$-adic and a $\ci$ incarnation.  In the scalar-weight, $\ci$-case, these operators agree with ones studied by Shimura.  In the final section of the paper, we also discuss some generalizations to other groups and settings.

The results from this paper apply to the author's paper-in-preparation with J. Fintzen, E. Mantovan, and I. Varma and to her ongoing joint project with M. Harris, J. -S. Li, and C. Skinner; they also relate to her recent paper with X. Wan.
\end{abstract}

\setcounter{tocdepth}{2}
\tableofcontents
\setcounter{secnumdepth}{3}

\section{Introduction}
\subsection{Overview}

This paper has two main parts.  First, we construct certain differential operators, which generalize operators studied by G. Shimura.  Then, as an application of some of these differential operators, we construct certain $p$-adic families of automorphic forms.  Beyond this paper, the anticipated applications are to the construction of $p$-adic $L$-functions (for example, the author's ongoing joint work with M. Harris, J.-S. Li, and C. Skinner), in analogue with the way that analogous differential operators in \cite{sh, shar} are used in the context of ($\IC$-valued) $L$-functions.  The present paper also sets the stage and introduces some of the main ideas for the author's joint paper-in-preparation with J. Fintzen, E. Mantovan, and I. Varma (constructing other $p$-adic families of automorphic forms).  In separate work with X. Wan concerning Klingen Eisenstein series (\cite{EW}), we discuss related differential operators and apply them to obtain results on Klingen Eisenstein series.  These operators and families are also (presumably) related to ones being studied by Skinner and E. Urban in the case of Siegel modular forms (a project they mentioned in 2010 at the RTG/FRG conference at UCLA); it will be interesting to see how the two approaches relate to each other.

This paper is intended to be a (relatively) concise and natural application of the author's prior results from \cite{EDiffOps, apptoSHL, apptoSHLvv}.  In the course of the paper, we also explain how to extend certain definitions and results on the differential operators developed by Shimura.  (Shimura's motivation for studying these differential operators, namely studying the algebraicity of values of $L$-functions, is discussed the second to last paragraph on \cite[p. 11137]{ShPNAS}, as well as in \cite[Introduction and Section 22]{sh}.)

\subsubsection{Organization of the paper}
In Section \ref{autformsunitarygroups-section}, we briefly review the basic theory of automorphic forms on unitary groups and introduce our setup.  In Section \ref{diffops-section}, we introduce certain differential operators, first in the $\ci$ setting and then in the $\padic$ setting.  These operators are built from the differential operators in \cite{EDiffOps}.  The differential operators in the present paper map automorphic forms on a unitary group of signature $(n,n)$ to automorphic forms on the product $U^\varphi\times U^{-\varphi}$ of two unitary groups, where $U^\varphi$ denotes the unitary group associated to a Hermitian form $\varphi$ of arbitrary signature on an $n$-dimensional vector space.  In the scalar-weight, $\ci$-case, these operators agree with ones studied by Shimura in \cite[Section 23]{sh}.  In Section \ref{measure-section}, we apply some of these differential operators to the Siegel Eisenstein series introduced in \cite{apptoSHLvv} to obtain certain $p$-adic automorphic form-valued measures; this is the main result of the paper.  We conclude the paper, in Section \ref{othergroupspullbacks}, by discussing generalizations to other groups.

\subsection{Conventions}

Fix a CM field $K$ with maximal totally real subfield $E$.  Fix a rational prime $p$ that is unramified in $K$ and such that every prime in $E$ above $p$ splits completely in $K$.  Fix embeddings
\begin{align*}
\iota_\infty: &\bar{\IQ}\hookrightarrow \IC\\
\iota_p:&\bar{\IQ}\hookrightarrow\IC_p.
\end{align*}
Fix an isomorphism
\begin{align*}
\iota: \IC_p\isomto\IC
\end{align*}
satisfying $\iota\circ\iota_p = \iota_\infty$.  From here on, we identify $\bar{\IQ}$ with $\iota_p(\bar{\IQ})$ and $\iota_\infty(\bar{\IQ})$.

Fix a CM type $\Sigma\subseteq\Hom(K, \bar{\IQ})$.  
Given $\sigma\in\Hom(E, \bar{\IQ})$, we also write $\sigma$ to denote the unique element of $\Sigma$ prolonging $\sigma$ (whenever no confusion can arise).  Let $\epsilon$ be the unique non-trivial element in $\Gal(K/E)$.  For any $x\in K$ and $\sigma\in \Sigma$, we define
\begin{align*}
\bar{x}&:=\epsilon(x)\\
\bar{\sigma} &:= \sigma\circ\epsilon.
\end{align*}
For each $\sigma\in\Sigma$, let $\mathfrak{p}_\sigma$ be the prime in $K$ over $p$ determined by $\iota_p\circ\sigma$.
 The correspondence $\sigma\leftrightarrow \mathfrak{p}_\sigma$ determines a set $\Sigma_p$ containing exactly one place in $K$ above each place in $E$ above $p$.  We identify $\sigma$ with $\mathfrak{p}_\sigma$ and $\Sigma$ with $\Sigma_p$ via $\sigma\leftrightarrow \mathfrak{p}_\sigma$ throughout.  Given a place $w$ of $\Sigma_p$, we also write $\sigma_w$ to denote the element of $\Sigma$ determining $w$.

For any field $F$, we write $\mathcal{O}_F$ to denote the ring of integers in $F$.  We denote by $\adeles$  the ad\`eles over $\IQ$, and we denote by $\adeles^{(\infty)}$ the ad\`eles at the finite places of $\IQ$.  Similarly, we denote by $\adeles_F$  the ad\`eles over $F$ for any field $F$, and we denote by $\adeles_F^{(\infty)}$ the ad\`eles at the finite places of $F$.  Given a modulus $\mathfrak{m}$, we denote by $\adeles_F^{(\mathfrak{m})}$ the ad\`eles at the places $v$ of $F$ not dividing $\mathfrak{m}$; when $F = \IQ$, we omit the subscript $F$.  For any finite extension of fields $L/F$, we write $\mathbf{N}_{L/F}$ to denote the norm from $L$ to $F$.  Given a $\mathcal{O}_F$-algebra $R$, the norm map $\mathbf{N}_{L/F}$ on $L$ determines a group homomorphism
\begin{align*}
\left(\mathcal{O}_L\otimes R\right)^\times \rightarrow R^\times
\end{align*}
via $a\otimes r\mapsto\mathbf{N}_{L/F}(a)r$.  When the fields are clear from context, we simply write $\mathbf{N}$.

Given an indexing set $X$, a ring $R$, an element $a$ of $R$, and a tuple $d = \left(d_x\right)_{x\in X}$ of exponents in $\IC$ such that $a^{d_x}$ is an element of $R$, we denote by $a^d$ the element defined by
\begin{align*}
a^d:=\prod_{x\in X}a^{d_x}\in R.
\end{align*}
If $a = \left(a_x\right)_{x\in X}\in R^X$, we write
\begin{align*}
a^d:=\prod_{x\in X}a_x^{d_x}\in R.
\end{align*}   Also, if $m = \left(m_x\right)_{x\in X}$, we denote by $m+d$ the element $\left(m_x+d_x\right)_{x\in X}$.  If $m$ is instead an element of $\IC$, we write $m+d$ or $d+m$ the element $\left(m+d_x\right)_{x\in X}$.

We denote by $M_{r\times s}(R)$ the $r\times s$ matrices with entries in $R$ for any ring $R$, and we let $\hern(R)$ denote the Hermitian matrices in $M_{n\times n}(R)$ (for appropriate $R$).  Given $a\in \hern(R)$, we write $a>0$ if $a$ is positive definite.  We denote by ${ }^t a$ the transpose of $a$.

\section{Automorphic forms on unitary groups}\label{autformsunitarygroups-section}

\subsection{Unitary groups}\label{unitarygroupssection}

\subsubsection{Preliminaries and conventions for unitary groups}
Given a finite-dimensional vector space $\mathcal{V}$ over the CM field $K$, a function $\phi: \mathcal{V}\times \mathcal{V}\rightarrow K$, and $g\in \End_K(\mathcal{V})$, we denote by $g\cdot f$ the function defined by
\begin{align*}
(g\cdot f)\left(v, w\right) := f\left(g\left(v\right), g\left(w\right)\right)
\end{align*}
for all $v, w \in \mathcal{V}$.  For any nondegenerate Hermitian form $\phi$ on $\mathcal{V}$, $\phi$ extends linearly to a function on $\mathcal{V}\otimes_\IQ R$ for all $\IQ$-algebras $R$; we denote by
\begin{align*}
U^\phi = U(\phi)= U(\mathcal{V}, \phi)
\end{align*}
the algebraic group whose $R$-points, for any $\IQ$-algebra $R$, are given by
\begin{align*}
U^\phi(R) = \left\{g\in\End_{K\otimes_\IQ R}\left(\mathcal{V}\otimes_\IQ R\right)|g\cdot \phi = \phi\right\}.
\end{align*}
We also denote by $GU^\phi = GU(\phi) = GU(\mathcal{V}, \phi)$ the algebraic group whose $R$-points, for any $E$-algebra $R$, are given by
\begin{align*}
GU^\phi(R) = \left\{g\in\End_{K\otimes_\IQ R}\left(\mathcal{V}\otimes_\IQ R\right)|g\cdot \phi = \lambda\phi \mbox{ with } \lambda\in R^\times\right\}.
\end{align*}
We denote by $\nu$ the similitude character
\begin{align*}
\nu: GU^\phi\rightarrow \mathrm{Res}_{E/\IQ}\gm
\end{align*}
defined by
\begin{align*}
g\cdot\phi = \nu(g)\phi.
\end{align*}
When $R = \adeles_E$ or $R = \IR$, we define
\begin{align*}
GU^\phi_+(R)
\end{align*}
to be the subgroup of $GU^\phi$ consisting of elements such that the similitude factor at each archimedean place of $K\otimes_E R$ is positive.
Note that the choice of a CM type $\Sigma$ induces a natural isomorphism
\begin{align*}
\mathcal{V}\otimes_\IQ\IR\isomto \prod_{\sigma\in\Sigma}\mathcal{V}_\sigma,
\end{align*}
where $\mathcal{V}_\sigma$ is a $\dim_K\mathcal{V}$-dimensional $\IC$-vector space.  So $U^\phi(\IR)$ is isomorphic to a subgroup
\begin{align}\label{signatureisom}
\prod_{\sigma\in\Sigma}U^\phi_\sigma\subseteq\prod_{\sigma\in\Sigma}\gl_{\dim_K\mathcal{V}}(\IC).
\end{align}

\subsubsection{A choice of unitary groups}
Fix a positive integer $n$.  Let $V$ be an $n$-dimensional vector space over the CM field $K$ together with a nondegenerate Hermitian form $\varphi$.  Let $-V$ denote the vector space $V$ together with the Hermitian form $-\varphi$, where $-\varphi$ denotes the map
\begin{align*}
V\times V&\rightarrow K\\
(u, v)&\mapsto -\varphi(u, v).
\end{align*}
Let 
\begin{align}\label{equ-etadef}
\eta = \varphi\oplus-\varphi
\end{align}
be the Hermitian form on $W = V\oplus -V$ defined by
\begin{align*}
\eta\left(\left(v_1, w_1\right), \left(v_2, w_2\right)\right) = \varphi(v_1, v_2)-\varphi(w_1, w_2),
\end{align*}
for all $v_1, w_1, v_2, w_2\in V$.  Observe that the canonical map
\begin{align*}
V\oplus-V&\isomto W = V\oplus -V\\
(u, v)&\mapsto (u, v),
\end{align*}
together with the definition of $\eta$, induces embeddings
\begin{align*}
U^\varphi\times U^{-\varphi}&\hookrightarrow U^\eta\\
G(U^\varphi\times U^{-\varphi})&\hookrightarrow GU^\eta,
\end{align*}
where
\begin{align*}
G(U^\varphi\times U^{-\varphi}) = \left\{(g, h)\in U^\varphi\times U^{-\varphi} | \nu(g) = \nu(h)\right\}.
\end{align*}
As in Expression \eqref{signatureisom}, we write
\begin{align*}
U^\varphi(\IR)\cong \prod_{\sigma\in\Sigma}U^\varphi_\sigma\subseteq\prod_{\sigma\in\Sigma}\gln(\IC)\\
U^\eta(\IR)\cong \prod_{\sigma\in\Sigma}U^\eta_\sigma\subseteq\prod_{\sigma\in\Sigma}\gln(\IC).
\end{align*}
Let $(a_\sigma, b_\sigma)$ be the signature of $U^\varphi_\sigma$.  Then the signature of $U^{-\varphi}_\sigma$ is
\begin{align*}
\left(b_\sigma, a_\sigma\right)_{\sigma\in\Sigma}\in\left(\ZZ_{\geq0}\times\ZZ_{\geq0}\right)^\Sigma,
\end{align*}
and the signature of $U_\sigma^{\eta}$ is $(n,n)$ for all $\sigma\in\Sigma$.  
Let
\begin{align*}
a &= (a_\sigma)_{\sigma\in\Sigma}\in \ZZ_{\geq0}^\Sigma\\
b&=(b_\sigma)_{\sigma\in\Sigma}\in \ZZ_{\geq0}^\Sigma.
\end{align*}

\subsection{Hermitian symmetric spaces and moduli spaces}\label{symmspacesshvarieties}

Note that for any nondegenerate Hermitian form $\phi$ on a finite-dimensional vector space $\V$ over $K$, $U^\phi(\IR)$ modulo a maximal compact subgroup $\cpct^\phi$ is a Hermitian symmetric space $\mathfrak{Z}^\phi$.  Shimura has written down such symmetric spaces explicitly in terms of coordinates (for instance, in \cite[Equation (26.4)]{shar} and \cite[Equation (6.1.4)]{sh}).  In \cite[Section 5.2]{hida}, H. Hida also provides an explicit description of these domains in terms of coordinates.  Our particular results and proofs, like those in \cite{ShPNAS}, however, do not require the details of these spaces in terms of coordinates; so to keep our current discussion as simple and clear as possible, we omit their details (except in the case where $\phi$ is of signature $(n,n)$) in this paper.  When the signature of $U^\phi_\sigma$ is $(n,n)$ for all $\sigma\in\Sigma$, the symmetric space $\mathfrak{Z}^\phi$ is holomorphically isomorphic to the space
\begin{align*}
\hn = \prod_{\sigma\in\Sigma}\h{n, \sigma},
\end{align*}
where
\begin{align*}
\h{n,\sigma}=\left\{z\in\gln(\IC)\mid i\left({ }^t\bar{z}-z\right)>0\right\}
\end{align*}
for all $\sigma$ in $\Sigma$.

Given maximal compact subgroups $\cpct^\varphi$ of $U^\varphi(\IR)$ and $\cpct^{-\varphi}$ of $U^{-\varphi}(\IR)$, respectively, we shall always choose a maximal compact subgroup $\cpct^{\varphi\oplus-\varphi}$ of $U^\eta(\IR) = U^{\varphi\oplus-\varphi}(\IR)$ so that
\begin{align*}
\cpct^{\varphi\oplus-\varphi}\cap\left(U^\varphi(\IR)\times U^{-\varphi}(\IR)\right) = \cpct^\varphi\times\cpct^{-\varphi}.
\end{align*}
Note that there is then an inclusion
\begin{align}\label{symmspaceinclusion}
\iota_{\varphi, \infty}: \mathfrak{Z}^\varphi\times\mathfrak{Z}^{-\varphi}\hookrightarrow\hn
\end{align}
compatible with the inclusion
\begin{align*}
\iota_\varphi:U^\varphi\times U^{-\varphi}\hookrightarrow U^\eta.
\end{align*}
\begin{rmk}
In \cite[Equation (6.10.2)]{sh}, Shimura gives such an inclusion $\iota_{\varphi, \infty}$ (denoted $\iota$) explicitly in terms of his choice of coordinates for the symmetric spaces; this particular map is holomorphic in the first variable and antiholomorphic in the second variable.  In \cite[A8.13]{shar}, though, he modifies the choice so that the map is holomorphic in both variables.
\end{rmk}

For any nonnegative integer $m$, let 
\begin{align*}
U(m):=\left\{g\in\gln(\IC)\mid g{ }^t\bar{g} = 1_m\right\}.
\end{align*} 
Note that we shall choose our maximal compact subgroups together with isomorphisms
\begin{align}
\cpct^{\varphi\oplus-\varphi}&\cong \prod_{\sigma\in\Sigma}\left(U(n)\times U(n)\right)\label{uniso}\\
\cpct^\varphi&\cong \prod_{\sigma\in\Sigma}\left(U\left(a_\sigma\right)\times U\left(b_\sigma\right)\right)\\
\cpct^{-\varphi}&\cong \prod_{\sigma\in\Sigma}\left(U\left(b_\sigma\right)\times U\left(a_\sigma\right)\right),
\end{align}
so that the inclusion
\begin{align}\label{cpctinclusion}
\iota_{\cpct}:\cpct^\varphi\times \cpct^{-\varphi}\hookrightarrow\cpct^{\varphi\oplus-\varphi}
\end{align}
together with these isomorphisms is then identified with the inclusion\begin{align*}
\prod_{\sigma\in\Sigma}\left(U\left(a_\sigma\right)\times U\left(b_\sigma\right)\right)\times\prod_{\sigma\in\Sigma}\left(U\left(b_\sigma\right)\times U\left(a_\sigma\right)\right)&\hookrightarrow \prod_{\sigma\in\Sigma}\left(U(n)\times U(n)\right)\\
\left(\left(g, h\right), \left(g', h'\right)\right)&\mapsto \left(\diag\left(g, g'\right), \diag\left(h, h'\right)\right).
\end{align*}
Note that by choosing a basis, we may identify the complexification of $\left(U(n)\times U(n)\right)$ with a subgroup of $\prod_{\sigma\in\Sigma}\left(\gln(\IC)\times\gln(\IC)\right)$.  Using the isomorphism \eqref{uniso}, we identify the complexification $\cpct^{c, \varphi\oplus-\varphi}$ of $\cpct^{\varphi\oplus-\varphi}$ with this subgroup of $\prod_{\sigma\in\Sigma}\left(\gln(\IC)\times \gln(\IC)\right)$.  Then the inclusions of the complexifications $\cpct^{c, \varphi}$ and $\cpct^{c, -\varphi}$ of $\cpct^{\varphi}$ and $\cpct^{-\varphi}$ into $\prod_{\sigma\in\sigma}\gl_{a_\sigma}(\IC)\times\gl_{b_\sigma}(\IC)$ and $\prod_{\sigma\in\Sigma}\gl_{b_\sigma}(\IC)\times\gl_{a_\sigma}(\IC)$, respectively, induced by $\iota_\cpct$ give a commutative diagram:
\begin{tiny}
\begin{align}\label{commdiag}
\xymatrix{
\cpct^{c, \varphi}\ar@{^{(}->}[dd]\times\cpct^{c, -\varphi}\ar@{^{(}->}[dd]\ar@{^{(}->}[rrrr]^{\mbox{the inclusion induced by } \iota_\cpct}  & && &\cpct^{c, \varphi\oplus-\varphi}\ar@{^{(}->}[dd]\\
& &&&&\\
 \prod_{\sigma\in\Sigma}\left(\gl_{a_\sigma}(\IC)\times\gl_{b_\sigma}(\IC)\right)\times\left(\gl_{b_\sigma}(\IC)\times\gl_{a_\sigma}(\IC)\right)\ar@{^{(}->}[rrrr]_{\hspace{2cm}\left(\left(g, h\right), \left(g', h'\right)\right)\mapsto\left(\diag\left(g, g'\right), \diag\left(h, h'\right)\right)}& &&&\prod_{\sigma\in\Sigma}\left(\gl_n(\IC)\times\gl_n(\IC)\right)
 }
\end{align}
\end{tiny}

We now generalize our discussion to the context of Shimura varieties.  We summarize the pertinent details here; a more detailed discussion is provided in \cite[Chapter 7]{hida}.  For any nondegenerate Hermitian form $\phi$ on a finite-dimensional vector space $\V$ over $K$, we use the notation
\begin{align*}
G^\phi:=GU^\phi.
\end{align*}
Fix a $G^\phi(\IR)$-conjugacy class $X^\phi$ of a fixed homomorphism of real algebraic groups
\begin{align*}
h_0: \IC^\times \hookrightarrow G^\phi
\end{align*}
over $\IR$.  The connected component of $X^\phi$ is isomorphic to the symmetric domain $G^\phi(\IR)/C^\phi_0$, where $C^\phi_0$ denotes the stabilizer of $h_0$ in $G^\phi(\IR)$, which is in turn isomorphic to copies of the symmetric domains $\mathfrak{Z}^\phi$ discussed above.

To the data $(G^\phi,X^\phi)$, as explained in in \cite[Chapter 7]{hida}, one may attach a Shimura variety\footnote{Actually, as noted in \cite[Section (1.2)]{HLS}, the data $\left(G^\phi, X^\phi\right)$ satisfies the axioms of \cite{deshimura} necessary for defining a Shimura variety, so long as $G^\phi(\IR)$ is not definite.  If $G^\phi(\IR)$ is definite, then one may attach a zero-dimensional Shimura variety to $\left(G^\phi, X^\phi\right)$, as explained in \cite{harriscrelle}.} $\Sh\left(G^\phi, X^\phi\right)$, whose $\IC$-points are given by
\begin{align*}
\varprojlim_{\mathcal{K}} G^\phi(\IQ)\backslash G^\phi(\adeles)/\cpct C^\phi_0,
\end{align*} 
where the projective limit runs over all open compact subgroups $\cpct$ of $G^\phi\left(\adeles^{(\infty)}\right)$.  As noted on \cite[p. 323]{hida}, the projective limit can be identified set-theoretically with
\begin{align*}
G^\phi(\adeles)\backslash\left(X^\phi\times G^\phi\left(\adeles^{(\infty)}\right)\right)/\overline{Z^\phi\left(\IQ\right)},
\end{align*}
where $Z^\phi$ denotes the center of $G^\phi$ and $\overline{Z^\phi\left(\IQ\right)}$ denotes the closure of $Z^\phi(\IQ)$ in $G^\phi\left(\adeles^{(\infty)}\right)$.  (The action\footnote{There is a typo in the description of this action given at the bottom of \cite[p. 323]{hida}.  We have corrected it here.} is given by $\gamma(x, g)z = (\gamma x, \gamma g z)$ for all $\gamma\in G^\phi(\IQ)$ and $z\in Z^\phi(\IQ)$.)  Also, note that
\begin{align*}
G^\phi(\IQ)\backslash G^\phi(\adeles)/\cpct C_0 = G(\IQ)\backslash\left(X^\phi\times G^\phi\left(\adeles^{(\infty)}\right)\right)/\cpct
\end{align*}
for each open compact subgroup $\cpct$ of $G^\phi\left(\adeles^{(\infty)}\right)$.  As explained on \cite[p. 317]{hida}, the connected component of the complex analytic space
\begin{align*}
\Sh_\cpct\left(G^\phi, X^\phi\right)(\IC):=G^\phi(\IQ)\backslash\left(X^\phi\times G^\phi\left(\adeles^{(\infty)}\right)\right)/\cpct
\end{align*}
is isomorphic to
\begin{align*}
\Gamma\backslash \mathfrak{Z}^\phi,
\end{align*}
where $\Gamma$ is the congruence subgroup of $G^\phi$ defined by 
\begin{align*}
\Gamma=\left(g\cpct g^{-1}G^\phi_+(\IR)\right)\cap G^\phi(\IQ)
\end{align*}
for an appropriate choice of $g\in G^\phi\left(\adeles^{(\infty)}\right)$.  

The Shimura variety $\Sh\left(G^\phi, X^\phi\right)$ represents a PEL moduli problem; the reader may consult any of the following references for details:  \cite[Chapter 7]{hida}, \cite[Chapters 1 and 2]{la},  \cite[Chapters 6 and 8]{milne}, and \cite[Section 2]{EDiffOps}.

Let 
\begin{align*}
G^{\phi, -\phi} = G\left(U^\phi\times U^{-\phi}\right).
\end{align*}
According to \cite[Section (1.2)]{HLS}, one can similarly define a Shimura datum $\left(G^{\phi, -\phi}, X^{\phi, -\phi}\right)$ so that there is a map of Shimura varieties
\begin{align*}
\Sh\left(G^{\phi, -\phi}, X^{\phi, -\phi}\right)\rightarrow\Sh\left(G^{\phi\oplus-\phi}, X^{\phi\oplus-\phi}\right).
\end{align*}

Slightly more generally, we may consider the following situation.  Suppose that $\cpct$ and $\cpct^{(p)}$ are {\it neat}, as defined in \cite[Definition 1.4.1.8]{la}.  Then the complex analytic spaces
\begin{align*}
\Sh_\cpct\left(G^\phi, X^\phi\right)(\IC)&:=G^\phi(\IQ)\backslash\left(X^\phi\times G^\phi\left(\adeles^{(\infty)}\right)\right)/\cpct\\
\Sh^{(p)}_\cpct\left(G^\phi, X^\phi\right)(\IC)&:=G^\phi(\IQ)\backslash\left(X^\phi\times G^\phi\left(\adeles^{(\infty)}\right)\right)/\tilde{G}^\phi_p\left(\ZZ_p\right)\cpct^{(p)},
\end{align*}
with $\tilde{G}^\phi_p\left(\ZZ_p\right)$ a hyperspecial maximal compact subgroup of $G^\phi(\IQ_p)$, are the complex points of smooth schemes $\Sh_\cpct\left(G^\phi, X^\phi\right)$ and $\Sh^{(p)}_\cpct\left(G^\phi, X^\phi\right)$, respectively, which represent a PEL moduli problem, i.e. they are moduli spaces for certain abelian schemes of PEL type (discussed in \cite[Chapter 7]{hida}, \cite[Chapters 1 and 2]{la},  \cite[Chapters 6 and 8]{milne}, and \cite[Section 2]{EDiffOps}).  These schemes are defined over $F$ and $\left(\mathcal{O}_F\right)_{(p)}$, respectively, where $F$ denotes the reflex field of $\left(G^\phi, X^\phi\right)$, as defined in \cite[Definition 1.2.5.4]{la} and \cite[Section 7.1.1]{hida}.
If $\cpct$ is a compact subgroup of $G^{\varphi\oplus-\varphi}(\adeles^{(\infty)})$ such that $\cpct\cap \left(G^{\varphi}(\adeles^{(\infty)})\times G^{-\varphi}(\adeles^{(\infty)})\right) =\cpct_1\times \cpct_2$, then there are inclusions
\begin{align}
\Sh_{\cpct_1}\left(G^\varphi, X^\varphi\right)\times\Sh_{\cpct_2}\left(G^{-\varphi}, X^{-\varphi}\right)&\hookrightarrow\Sh_\cpct\left(G^{\varphi\oplus-\varphi}, X^{\varphi\oplus-\varphi}\right)\label{avpullbk}\\
\Sh^{(p)}_{\cpct_1}\left(G^\varphi, X^\varphi\right)\times\Sh^{(p)}_{\cpct_2}\left(G^{-\varphi}, X^{-\varphi}\right)&\hookrightarrow\Sh^{(p)}_\cpct\left(G^{\varphi\oplus-\varphi}, X^{\varphi\oplus-\varphi}\right)\label{avpullbk2}
\end{align}
compatible with $\iota_\mathcal{K}$ and $\iota_{\varphi, \infty}$.

Over the ordinary locus $S_\phi$ of any of the above moduli spaces $M_\phi=\Sh^{(p)}_\cpct\left(G^\phi, X^\phi\right)$ of abelian schemes of PEL type, there is the Igusa tower, a formal scheme over $\mathcal{O}_F\otimes\ZZ_p$ that parametrizes ordinary abelian varieties (together with additional structure, namely a polarization, endomorphism, and level structure); and we have analogous pullbacks $\iota_{\varphi, p}$ in this case, compatible with the inclusions of groups $\iota_\varphi$ given above.  The reader can find a construction of the Igusa tower in \cite[Section 8.1]{hida}.  

\subsection{Automorphic forms}\label{autoforms-section}

Note that there is a natural action of $G^\phi(\IR)$ on $\mathfrak{Z}^\phi$ (described explicitly in terms of coordinates in \cite[Section 6.3]{sh}).  Given $\alpha\in G^\phi(\IR)$ and $z\in\mathfrak{Z}^\phi$, we write
\begin{align*}
\alpha z = \left(\alpha_vz_v\right)_{v\in\Sigma}
\end{align*}
to denote the action of $\alpha$ on $z$.  Note that there are also canonical automorphy factors (also described explicitly in terms of coordinates in \cite[Section 6.3]{sh}))
\begin{align*}
\mu_{\alpha}(z)&=\mu(\alpha, z)=\left(\mu(\alpha, z)\right)_{v\in\Sigma}\\
\lambda_\alpha(z)&=\lambda(\alpha, z)=\left(\lambda(\alpha, z)\right)_{v\in\Sigma},
\end{align*}
with
\begin{align*}
M_{\alpha}(z):=\left(\mu_{\alpha}(z), \lambda_{\alpha}(z)\right)\in \cpct^\phi(\IC).
\end{align*}
Given a Hermitian form $\phi$ as above, a complex vector space $\mathcal{W}$ together with a representation $\rho$ of $\cpct^\phi(\IC)$ on $\mathcal{W}$, and a function
\begin{align}\label{autform}
f: \mathfrak{Z}^\phi\rightarrow \mathcal{W},
\end{align}
we define (using the notation of \cite[Equation (12.6)]{shar})
\begin{align*}
\left(f||_{\rho}\alpha\right)(z):= \rho\left(M_{\alpha}(z)\right)^{-1}f(\alpha z).
\end{align*}
If $\mathcal{W}$ is one-dimensional and $(k, \nu) = \left(k(\sigma), \nu(\sigma)\right)_{\sigma\in\Sigma}\in \left(\ZZ\times\ZZ\right)^{\Sigma}$ is such that
 \begin{align}\label{autfactorrho}
 \rho\left(M_{\alpha}(z)\right) = \det\left(\mu_{\alpha}(z)\right)^{k+\nu} \det\left(\lambda_\alpha(z)\right)^{-\nu},
 \end{align}
 for all $\alpha$, then we define
 \begin{align*}
 f||_{k, \nu} := f||_{\rho}.
 \end{align*}
When 
\begin{align}\label{symmspaceautomorphy}
f||_{\rho}\alpha = f
\end{align}
 for all $\alpha$ in some congruence subgroup $\Gamma$ of $G^{\phi}$, we say that $f$ is a {\it $\mathcal{W}$-valued automorphic form of weight $\rho$} and level $\Gamma$ on $\mathfrak{Z}^\phi$.  If
 \begin{align*}
 f||_{k, \nu}\alpha = f,
 \end{align*}
 for all $\alpha$ in some congruence subgroup $\Gamma$ of $G^{\phi}$, we say that $f$ is a {\it $\IC$-valued automorphic form of weight $(k, \nu)$} and level $\Gamma$ on $\mathfrak{Z}^\phi$. 

Note that if $\Gamma_1$ and $\Gamma_2$ are congruence subgroups of $G^\varphi$ and $G^{-\varphi}$, respectively, such that
\begin{align*}
\Gamma_1\times \Gamma_2 = \Gamma\cap\left(G^\varphi\times G^{-\varphi}\right)
\end{align*}
 and $f$ is a $\mathcal{W}$-valued automophic form of weight $\rho$ and level $\Gamma$ on $\hn$, then
 \begin{align*}
 f\circ\iota_{\varphi, \infty}: \mathfrak{Z}^\varphi\times\mathfrak{Z}^{-\varphi}&\rightarrow \mathcal{W}\\
 (z, w)&\mapsto  f\circ\iota_{\varphi, \infty}(z, w)
 \end{align*}
 satisfies
 \begin{align*}
 \left( f\circ\iota_{\varphi, \infty}\right)||_{\left(\rho|_{\cpct^\varphi(\IC)\times\cpct^{-\varphi}(\IC)}\right)}\left(\alpha, \beta\right) = f\circ\iota_{\varphi, \infty}
 \end{align*}
 for all $(\alpha, \beta)$ in $\Gamma_1\times \Gamma_2\subseteq\Gamma$.  In particular, the pullback $f\circ\iota_{\varphi, \infty}$ of $f$ to $\mathfrak{Z}^\varphi\times \mathfrak{Z}^{-\varphi}$ is an automorphic form in each of the variables $z$ and $w$.

Given a representation $\left(\omega, \mathcal{U}\right)$ of $\cpct^\varphi(\IC)\times \cpct^{-\varphi}(\IC)$ and
\begin{align*}
f: \mathfrak{Z}^\varphi\times\mathfrak{Z}^{-\varphi}\rightarrow \mathcal{U},
\end{align*}
we define
\begin{align*}
f||_\omega: \mathfrak{Z}^\varphi\times\mathfrak{Z}^{-\varphi}&\rightarrow \mathcal{U}\\
(z, w)&\mapsto \omega\left(M_{\alpha}(z), M_\beta(w)\right)f(\alpha z, \beta w).
\end{align*}

Now, suppose $\left(\rho', \mathcal{W}'\right)$ is a subrepresentation of $\left(\rho|_{\cpct^\varphi(\IC)\times\cpct^{-\varphi}(\IC)}, \mathcal{W}\right)$.  Choose a projection $\pi_{\mathcal{W}'}$ onto the subspace $\mathcal{W}'$.  (Note that this projection might not be unique, since $\mathcal{W}'$ might not be of multiplicity $1$ in $\mathcal{W}$.)  Then
\begin{align*}
\left(  \pi_{\mathcal{W}'}\circ f\circ\iota_{\varphi, \infty}\right)||_{\rho'}\left(\alpha, \beta\right) =\pi_{\mathcal{W}'}\circ f\circ\iota_{\varphi, \infty}
 \end{align*}
 for all $(\alpha, \beta)$ in $\Gamma_1\times \Gamma_2\subseteq\Gamma$.  In particular, 
\begin{align*}
\pi_{\mathcal{W}'}\circ f\circ\iota_{\varphi, \infty}: \mathfrak{Z}^\varphi\times \mathfrak{Z}^{-\varphi}\rightarrow \mathcal{W}'
\end{align*}
is an automorphic form in each of the variables $z$ and $w$.  Note that $\pi_{\mathcal{W}'}\circ f$ is not necessarily an automorphic form on the larger space $\hn$.

Alternatively, one can define a $\mathcal{W}$-valued automorphic form on $G^\phi$ to be a function on $G^\phi$ satisfying an automorphy condition compatible with the one given in Equation \eqref{symmspaceautomorphy}.  This is the approach taken in \cite[A8.2]{shar}.  It is also the approach taken in parts of \cite{apptoSHL, apptoSHLvv}, which contain some of the main results upon which our proofs in the current paper rely; we do not, however, need the details of this perspective in the current paper.  In this context, we once again have pullbacks of automorphic forms and can define natural analogues of $f\circ\iota_{\varphi, \infty}$ and $\pi_{\mathcal{W}'}\circ f\circ\iota_{\varphi, \infty}$ compatible with the inclusions $\iota_{\varphi}$ and $\iota_\cpct$.

For a more general treatment of automorphic forms, it is most convenient to define automorphic forms as functions on moduli spaces of abelian varieties of PEL type, i.e. as sections of vector bundles over the moduli spaces $M_\phi:=\Sh_\cpct\left(G^\phi, X^\phi\right)$ (or $\Sh_\cpct^{(p)}\left(G^\phi, X^\phi\right)$).  
This is the approach of much of \cite{EDiffOps}, which gives a detailed treatment of automorphic forms from this perspective and the connection with the definition (from the perspective of symmetric spaces) used above.  Note that this is a natural setting in which to consider automorphic forms over rings other than $\IC$, as well as $\IC$; in particular, it is a more natural setting in which to study $p$-adic properties of automorphic forms.

Similarly to the above settings, we can also discuss pullbacks in this context; in this setting, we have that if $f$ is an automorphic form on $\Sh_\cpct\left(G^{\varphi\oplus-\varphi}, X^{\varphi\oplus-\varphi}\right)$ (or $\Sh^{(p)}_\cpct\left(G^{\varphi\oplus-\varphi}, X^{\varphi\oplus-\varphi}\right)$), then the pullback of $f$ under the map in Expression \eqref{avpullbk} (or \eqref{avpullbk2}) is an automorphic form.  So we may define natural analogues $f\circ\iota_{\Sh_{\cpct_1}\times\Sh_{\cpct_2}}$ and $\pi_{\mathcal{W}'}\circ f\circ\iota_{\Sh_{\cpct_1}\times\Sh_{\cpct_2}}$ of $f\circ\iota_{\varphi, \infty}$ and $\pi_{\mathcal{W}'}\circ f\circ\iota_{\varphi, \infty}$, respectively, where $\iota_{\Sh_{\cpct_1}\times\Sh_{\cpct_2}}$ denotes either of the two inclusions \eqref{avpullbk} and \eqref{avpullbk2}.  Note that $\pi_{\mathcal{W}'}\circ f\circ\iota_{\Sh_{\cpct_1}\times\Sh_{\cpct_2}}$ takes as its input a point 
\begin{align*}
\left(\uA_1\times \uA_2, \left(\uo_1, \uo_2\right)\right),
\end{align*}
where $\uA_1$ and $\uA_2$ are (isogeny classes of) abelian varieties over a scheme $S$ over $F$ (or over a scheme $S$ over $\mathcal{O}_F\otimes\ZZ_{(p)}$), together with a polarization, level structure, and endomorphism corresponding to the moduli problems represented by $\Sh_{\cpct_1}\left(G^\varphi, X^\varphi\right)$ and $\Sh_{\cpct_2}\left(G^{-\varphi}, X^{-\varphi}\right)$ (or $\Sh_{\cpct_1}^{(p)}\left(G^\varphi, X^\varphi\right)$ and $\Sh_{\cpct_2}^{(p)}\left(G^{-\varphi}, X^{-\varphi}\right)$), and $\uo_1$ and $\uo_2$ are ordered choices of bases for $\uo_{\uA_1/S}$ and $\uo_{\uA_2/S}$, respectively.

Similarly, $p$-adic automorphic forms are certain sections over the Igusa tower, and we have analogous pullbacks $\iota_{\varphi, p}$ in this case, compatible with the inclusions of groups $\iota_\varphi$ given above.  The reader can find a detailed discussion of $p$-adic automorphic forms in \cite[Chapter 8]{hida} and \cite{HLS}.  Let $R$ be a $p$-adic ring.  In what follows, for each Hermitian form $\varphi$, we let $\mathcal{M}_\rho\left(\padic, U^\varphi, R\right)$ denote Hida's space of {\it false} automorphic forms of weight $\rho$ (\cite[p. 329]{hida}), viewed as sections of a vector bundle over the ordinary locus $S_\varphi$, and we define $\mathcal{M}_\rho\left(\padic, U^\varphi\times U^{-\varphi}, R\right)$ analogously.  We define $\mathcal{M}_\rho\left(\padic, U^\varphi\right) := \mathcal{M}_\rho\left(\padic, U^\varphi, \OCp\right)$ and $\mathcal{M}_\rho\left(\padic, U^\varphi\times U^{-\varphi}\right):=\mathcal{M}_\rho\left(\padic, U^\varphi\times U^{-\varphi}, \OCp\right)$.

\section{Differential operators}\label{diffops-section}
\subsection{$\ci$-differential operators}\label{cidiffops-section}

In \cite{EDiffOps}, we discussed certain weight-raising $\ci$-differential operators that act on $\ci$-automorphic forms on unitary groups of signature $(n,n)$, i.e. automorphic forms on groups isomorphic to $GU^{\varphi\oplus-\varphi}$.  These operators are constructed geometrically, using the Gauss-Manin connection and the Kodaira-Spencer morphism.  The operators act on automorphic forms on $GU^{\varphi\oplus-\varphi}$, viewed as sections of a vector bundle of automorphic forms over the complex points of a moduli space of abelian schemes of PEL type (i.e. the analytification of a Shimura variety).  

On the other hand, Shimura constructed $\ci$-differential operators that act on $\ci$-automorphic forms on unitary groups, viewed as functions on $\hn$, as in map \eqref{autform}, or as functions on unitary groups \cite{sh84, shclassical, sh, shar}.  As explained in \cite{EDiffOps}, Shimura's $\ci$-operators and the $\ci$-operators in \cite{EDiffOps} are the same (once we identify the complex space $\Sh_\cpct\left(G^{\varphi\oplus-\varphi}, X^{\varphi\oplus-\varphi}\right)(\IC)$ with the complex space $\Gamma\backslash\hn$,  as explained above).
                    
Let $\mathcal{M}_\rho\left(\ci, U^\eta\right)$ denote the space of $\ci$-automorphic forms of weight $\rho$ on $U^\eta$, where $\rho$ denotes a representation of the complexification of a maximal compact subgroup $\mathcal{K}$ of $U^{\varphi\oplus-\varphi}(\IR)$.  Similarly, if $\rho$ is instead a representation of the complexification of $\cpct_1\times\cpct_2$, we write $\mathcal{M}_{\rho}\left(\ci, U^\varphi\times U^{-\varphi}\right)$ to denote the space of $\ci$-automorphic forms of weight $\rho$ on $U^\varphi\times U^{-\varphi}$.  When $\rho$ is as in Equation \eqref{autfactorrho}, we define
\begin{align*}
\mathcal{M}_{k, \nu}\left(\ci, U^{\varphi\oplus-\varphi}\right) := \mathcal{M}_{\rho}\left(\ci, U^{\varphi\oplus -\varphi}\right).
\end{align*}
Also, if $\rho$ is a representation of the complexification of $\cpct^{\varphi\oplus-\varphi}\cap\left(U^\varphi\times U^{-\varphi}\right)$ such that
\begin{align*}
\rho\left(M_\alpha(z), M_\beta(w)\right) = \det\left(\mu_{\alpha}(z)\right)^{k+\nu} \det\left(\lambda_\alpha(z)\right)^{-\nu}\det\left(\mu_{\beta}(w)\right)^{k+\nu} \det\left(\lambda_\beta(w)\right)^{-\nu},
\end{align*}
then we define
\begin{align*}
\mathcal{M}_{k, \nu}\left(\ci, U^\varphi\times U^{-\varphi}\right) := \mathcal{M}_{\rho}\left(\ci, U^\varphi\times U^{-\varphi}\right).
\end{align*}
If instead
\begin{align*}
\rho\left(M_\alpha(z), M_\beta(w)\right) = \det\left(\mu_{\alpha}(z)\right)^{k+\nu} \det\left(\lambda_\alpha(z)\right)^{-\nu}\det\left(\mu_{\beta}(w)\right)^{k'+\nu'} \det\left(\lambda_\beta(w)\right)^{-\nu'},
\end{align*}
then we define
\begin{align*}
\mathcal{M}_{\left(k, \nu\right),\left(k', \nu'\right)}\left(\ci, U^\varphi\times U^{-\varphi}\right) := \mathcal{M}_{\rho}\left(\ci, U^\varphi\times U^{-\varphi}\right).
\end{align*}

Let $m =\left(m\left(\sigma\right)\right)_{\sigma\in\Sigma}\in\ZZ_{\geq 0}^\Sigma$.  In \cite[Sections 23.6-23.14]{sh} and \cite[Section 29]{shar}, Shimura uses the $\ci$ -differential operators mentioned above to construct operators
\begin{align}\label{shamk}
A_m^k \left(= \prod_{\sigma\in\Sigma}A_{m}^{k}\left(\sigma\right)\right): \left\{f: \hn\rightarrow \IC\right\}\longrightarrow \left\{f: \mathfrak{Z}^\varphi\times\mathfrak{Z}^{-\varphi}\rightarrow \IC\right\}
\end{align}
such that if $f$ is a $\IC$-valued $\ci$-automorphic form of weight $(k, \nu)$, then $A_m^kf = \prod_{\sigma\in\Sigma}A_{m}^{k}\left(\sigma\right)f$ is an element of $\mathcal{M}_{\left(\left(k+m, \nu\right),\left(k+m, \nu-m\right)\right)}\left(\ci, U^\varphi\times U^{-\varphi}\right)$.  Shimura applies these operators to certain Eisenstein series on $\hn$ to obtain automorphic forms with particular properties necessary to complete some of his proofs about zeta functions.

We now present an alternate construction of the operators $A_m^k$, via geometry, in the process of which we will create a natural generalization of the operators $A_m^k$.  

\begin{rmk}
We are purposely ambiguous here about whether we are viewing $\ci$-automorphic forms as functions on $\mathfrak{Z}^\eta$ or as sections of a vector bundle on a moduli space.  As explained in \cite{EDiffOps}, these two perspectives can be identified with each other, and the $\ci$-differential operators defined in \cite{EDiffOps} and below agree with each other under this identification.  So we can pass freely between these two perspectives in our discussion of weight-raising differential operators on $\ci$-automorphic forms.  

Furthermore, as explained in \cite[A8]{shar}, there is a natural identification of $\ci$automorphic forms on $\hn$ with $\ci$-automorphic forms viewed as functions on the unitary group; and the $\ci$-differential operators in each setting agree with each other under this identification.  So, in fact, we may actually pass naturally between all three settings.

Below, by abuse of notation, we shall write $\iota_\varphi$ to mean either the compatible inclusion $\iota_{\varphi, \infty}$ on symmetric spaces defined in \eqref{symmspaceinclusion}, or a corresponding compatible inclusion on moduli spaces, or a corresponding compatible inclusion of groups.
\end{rmk}

We now recall information from \cite[Section 4.1]{apptoSHLvv} that we will need below.  For each $\sigma\in\Sigma$, let $\mathcal{T}_n(\sigma)$ denote the tangent space of $\h{n,\sigma}$.  We identify $\mathcal{T}_n(\sigma)$ with $M_{n\times n}(\IC)$, for all $\sigma\in\Sigma$.  Define
\begin{align*}
\mathcal{T}_n &= \prod_{\sigma\in\Sigma}\mathcal{T}_n(\sigma).
\end{align*}
For each nonnegative integer $d$, we denote by $\mathfrak{S}_d\left(\mathcal{T}_n(\sigma)\right)$ the $\IC$-vector space of $\IC$-valued homogeneous polynomial functions on $\mathcal{T}_n(\sigma)=M_{n\times n}(\IC)$ (equality under our identification of $\mathcal{T}_n(\sigma)$ with $M_{n\times n}(\IC)$) of degree $d$.  So, for example, $\det^d\in\mathfrak{S}_{nd}\left(\mathcal{T}_n(\sigma)\right)$, for any $d\in \ZZ_{\geq 0}^\Sigma$.  (By $\det^d(z)$, we mean $\prod_{\sigma\in\Sigma}\det z^{d(\sigma)}$)  For each $d =\left(d(\sigma)\right)_{\sigma\in\Sigma}\in\ZZ_{\geq 0}$, we define
\begin{align*}
\mathfrak{S}_d\left(\mathcal{T}_n\right) = \otimes_{\sigma\in\Sigma}\mathfrak{S}_{d(\sigma)}\left(\mathcal{T}_n(\sigma)\right).
\end{align*}
Following the notation of \cite{shar}, we denote by $\tau^d$ the representation of $\prod_{\sigma\in\Sigma}\left(\gln(\IC)\times\gln(\IC)\right)$ on $\mathfrak{S}_d\left(\mathcal{T}_n\right)$ defined by
\begin{align*}
\tau^d\left(\alpha, \beta\right) g(z) := g\left({ }^t\alpha z \beta\right)
\end{align*}
for all $\left(\alpha, \beta\right)\in\prod_{\sigma\in\Sigma}\left(\gln(\IC)\times\gln(\IC)\right)$, $z\in \mathcal{T}_n$, and $g\in\mathfrak{S}_d\left(\mathcal{T}_n\right)$.

 Let $d = \left(d(\sigma)\right)_{\sigma\in\Sigma}$ be an element of $\ZZ_{\geq 0}^\Sigma$.  As explained in extensive detail in \cite{kaCM, shar, EDiffOps}, there are ``weight-raising'' $\ci$-differential operators
\begin{align}\label{weightraisingnn}
D_{\rho}^d\left(\ci\right): \mathcal{M}_\rho\left(\ci, U^{\varphi\oplus-\varphi}\right)\rightarrow \mathcal{M}_{\rho\otimes\tau^d}\left(\ci, U^{\varphi\oplus-\varphi}\right),
\end{align}
which generalize the Maass-Shimura operators from the case of modular forms to the case of automorphic forms on unitary groups.  We briefly outline the key ingredients in the construction (from \cite{EDiffOps}) of the differential operators $D_\rho^d\left(\ci\right)$, which extends the construction of the $\ci$ differential operators for Hilbert modular forms in \cite[Section 2.3]{kaCM} to the case of automorphic forms on unitary groups of signature $(n,n)$.  Fix a maximal compact subgroup $\cpct$, and consider the moduli space $M_\phi=\Sh^{(p)}_\cpct\left(G^\phi, X^\phi\right)$ of abelian schemes of PEL type, as above, viewed over a base scheme $S$.  Denote by $\Auniv$ the universal abelian variety over $M_\phi$.  Adhering to the usual conventions, define $\uo:=\pi_*\left(\Omega^1_{\Auniv/M_{\phi}}\right)$ and $\hdr:=\mathbf{R}^1\pi_*\left(\Omega^\bullet_{\Auniv/M_{\phi}}\right)$ (where $\Omega_{\Auniv/M_{\phi}}^i$ denotes the sheaf of invariant $n$-forms for each positive integer $i$).  There sheaf $\uo$ decomposes as $\uo=\oplus_{\sigma\in\Sigma}\uo_\sigma^+\oplus \uo_\sigma^-$.  The Kodaira-Spencer morphism gives an isomorphism 
\begin{align}\label{ks-iso}
\ks:\Omega^1_{M_\phi/S}\isomto \oplus_{\sigma\in\Sigma}\uo_\sigma^+\otimes\uo_\sigma^-.
\end{align}

  We identify $\uo$ with its image in $\hdr$ under the inclusion $\uo\hookrightarrow\hdr$ coming from the Hodge filtration.  The Gauss-Manin connection $\nabla:\hdr\rightarrow\hdr\otimes\Omega^1_{M_\phi/S}$ 
  extends via the Leibniz rule (i.e. the product rule) to a connection, also denoted by $\nabla$, $\nabla: {\hdr}^{\otimes e}\rightarrow{\hdr}^{\otimes e}\otimes\Omega^1_{M_\phi/S}$ for each positive integer $e$.
  Define $D:=\left(\id\otimes\ks\right)\circ\nabla$, where $\id$ denotes the identity map on ${\hdr}^{\otimes e}$.  The map $D$ decomposes over $\sigma\in \Sigma$ as a direct sum of maps
  \begin{align*}
D(\sigma): {\hdr}^{\otimes e}\rightarrow{\hdr}^{\otimes e}\otimes\uo_\sigma^+\otimes\uo_\sigma^-.
  \end{align*}
Under the identification $\uo_\sigma^+\otimes\uo_\sigma^-\subseteq \uo^{\otimes 2}\subseteq \hdr\otimes\hdr$, define $D(\ell\sigma)$ to be the composition of $D(\sigma)$ with itself $\ell$ times, for each nonnegative integer $\ell$.  Define $D^d:=\oplus_{\sigma\in\Sigma}D(d\sigma)$.

When $S=\Spec\IC$, consider the $\ci$-manifold $M_\phi\left(\ci\right)$ underlying the complex manifold consisting of $\IC$-valued points of $M_\phi$.  For each sheaf $\mathcal{F}$ on $M_\phi$, denote by $\mathcal{F}\left(\ci\right)$ the sheaf on $M_\phi\left(\ci\right)$ obtained by tensoring $\mathcal{F}$ with the $\ci$-structural sheaf on $M_\phi\left(\ci\right)$.  Over $M_\phi\left(\ci\right)$, there is the Hodge decomposition
\begin{align*}
\hdr\left(\ci\right) = \uo\left(\ci\right)\oplus\overline{\uo\left(\ci\right)},
\end{align*}
where the bar denotes complex conjugation.  Note that $D$ extends naturally to a map on ${\hdr}^{\otimes e}\left(\ci\right)$.  Define $\pi_{\ci}: \hdr\left(\ci\right)\rightarrow \uo\left(\ci\right)$ to be projection modulo $\overline{\uo\left(\ci\right)}$, and also denote by $\pi_{\ci}$ the induced map $\hdr\left(\ci\right)^{\otimes e}\rightarrow \uo\left(\ci\right)^{\otimes e}$.  Define 
\begin{align*}
D^d\left(\ci\right):=\pi_{\ci}\circ D^d.
\end{align*}
Let $\rho$ be an irreducible polynomial representation of the complexification of a maximal compact subgroup $\mathcal{K}$ of $U^{\varphi\oplus-\varphi}(\IR)$.  Then the sheaf of automorphic forms of weight $\rho$ can be identified with a subsheaf $\uo^\rho$ of $\uo^{\otimes e}$ for an appropriate integer $e$, and the restriction of $D^d\left(\ci\right)$ to $\uo^\rho\left(\ci\right)$ gives a map $D_\rho^d\left(\ci\right)$ from the space of $\ci$ automorphic forms of weight $\rho$ to the space of $\ci$ automorphic forms of weight $\rho\otimes\tau^d$.

\begin{defi}\label{pllbkopdefi} Let $d = \left(d(\sigma)\right)_{\sigma\in\Sigma}$ be an element of $\ZZ_{\geq 0}^\Sigma$.  Let $\rho$ be a polynomial representation of the complexification of a maximal compact subgroup $\mathcal{K}$ of $U^{\varphi\oplus-\varphi}(\IR)$, and let $\gamma$ be a subrepresentation of the restriction of $\rho\otimes \tau^d$ to the complexification of $\mathcal{K}\cap \left(U^\varphi\oplus U^{-\varphi}\right)$.   Fix a projection $\pi_\gamma$ from $\rho$ onto $\gamma$.  We define an operator
\begin{align*}
A_{\rho}^\gamma\left(\ci\right): \mathcal{M}_\rho\left(\ci, U^{\varphi\oplus-\varphi}\right)\rightarrow \mathcal{M}_{\gamma}\left(\ci, U^\varphi\times U^{-\varphi}\right)
\end{align*}
by
\begin{align*}
A_\rho^\gamma(\ci)(f):=\pi_\mathcal{\gamma}\circ \left(D_\rho^d\left(\ci\right) f\right)\circ \iota_{\varphi, \infty}.
\end{align*}
\end{defi}

\begin{warning}
The $\ci$-function $\pi_\mathcal{\gamma}\circ \left(D_\rho^d\left(\ci\right) f\right)$ is typically not an automorphic form on $U^{\varphi\oplus-\varphi}$, since $\gamma$ is typically not a representation of the complexification of $\mathcal{K}$.
\end{warning}

\subsection{Highest weight vectors and irreducible polynomial representations}\label{highwtirrpoly-section}

We now build upon the brief discussion of polynomial representations in \cite[Section 4.1]{apptoSHLvv}.  More extensive details about these representations are available in \cite[Section 2]{shclassical} and \cite[Sections 12.6 and 12.7]{shar}.  Given positive integers $j \leq l$ and an $l\times m$ matrix $a$, we denote by ${\det}_j(a)$ the determinant of the upper left $j\times j$ submatrix of $a$.  Each polynomial representation of $\gl_l(\IC)$ can be decomposed into a direct sum of irreducible representations $\rho$ of $\gl_l(\IC)$.  These irreducible representations are indexed by highest weights, i.e. there is a one-to-one correspondence between the irreducible polynomial representations of $\gl_l(\IC)$ and the ordered $l$-tuples $r_1\geq \cdots\geq r_l\geq 0$ of integers dependent upon $\rho$.  Each irreducible polynomial representation $\rho$ of $\gl_l(\IC)$ contains a unique eigenvector $p$ of highest weight $r_1\geq \cdots\geq r_l\geq 0$ (for a unique ordered $l$-tuple $r_1\geq \cdots r_l\geq 0$ of integers dependent upon $\rho$), i.e. a common eigenvector of the upper triangular matrices of $\gl_n(\IC)$ that satisfies
\begin{align}
\rho(a) p &= \prod_{j=1}^{l}{\det}_j(a)^{e_j} p\nonumber\\
e_j&= r_j-r_{j+1}, 1\leq j\leq l-1\label{ej1}\\
e_l & = r_l,\label{ej2}
\end{align}
for all $a$ in the subgroup of upper triangular matrices in $\gl_l(\IC)$. 

\subsection{Subrepresentations and restrictions of $\tau^d$}\label{subrepres}
As above, for each $\sigma\in\Sigma$, we identify $\left(\gl_{a_\sigma}\left(\IC\right)\times\gl_{b_\sigma}(\IC)\right)\times\left(\gl_{b_\sigma}\left(\IC\right)\times\gl_{a_\sigma}(\IC)\right)$ with its image in $\gln(\IC)\times\gln(\IC)$ under the embedding
 \begin{align*}
\left(\gl_{a_\sigma}\left(\IC\right)\times\gl_{b_\sigma}(\IC)\right)\times\left(\gl_{b_\sigma}\left(\IC\right)\times\gl_{a_\sigma}(\IC)\right)\hookrightarrow\prod_{\sigma\in\Sigma}\left(\gln(\IC)\times\gln(\IC)\right).
\end{align*}
given in the diagram \eqref{commdiag}.  Generalizing this slightly, we also consider the case in which $\IC$ is replaced by an arbitrary ring $S$.  (For us, $S$ will always be $\IC$ or a $p$-adic ring.)

We now discuss the irreducible subrepresentations of the representation $\tau^d$ of $\prod_{\sigma\in\Sigma}\left(\gln(\IC)\times\gln(\IC)\right)$, as well as the subrepresentations of the restriction of $\tau^d$ to 
the subgroup $\prod_{\sigma\in\Sigma}\left(\gl_{a_\sigma}\left(\IC\right)\times\gl_{b_\sigma}(\IC)\right)\times\left(\gl_{b_\sigma}\left(\IC\right)\times\gl_{a_\sigma}(\IC)\right)$ of $\prod_{\sigma\in\Sigma}\left(\gln(\IC)\times\gln(\IC)\right)$.

Generalizing our notation from above, for any positive integers $d$, $r$, and $s$ and any ring $R$ of $r\times s$ matrices over a ring $S$, we denote by $\mathfrak{S}_d\left(R\right)$ the $S$-module of all $S$-valued homogeneous polynomial functions on $R$ of degree $d$.  (Above, we only considered the case in which $r=s=n$, $S=\IC$, and $R = M_{n\times n}(\IC)$.)  

For all positive integers $r$, $s$ and $d$, we define a representation $\tau^d_{r, s}$ of $\gl_r(S)\times\gl_s(S)$ on the $S$-module $\mathfrak{S}_d\left(M_{r\times s}\left(S\right)\right)$ by
\begin{align*}
\tau^d\left(\alpha, \beta\right)g(z) = g\left({ }^t\alpha z \beta\right)
\end{align*}
for all $\alpha\in\gl_r(S)$, $\beta\in\gl_s(S)$, $z\in M_{r\times s}\left(S\right)$, and $g\in \mathfrak{S}_d\left(M_{r\times s}\left(S\right)\right)$.  Observe that
\begin{align*}
\tau^d = \otimes_{\sigma\in\Sigma}\tau^d_{n, n}.
\end{align*}

The following theorem, whose proof is due to \cite{hua, johnson} and appears in \cite{shclassical}, describes the decomposition of $\tau^d_{r, s}$ into irreducible subrepresentations.
\begin{thm}[Theorem 12.7 in \cite{shar}]\label{thm127}
Let $d$, $q$, and $s$ be positive integers.  Let $\rho$ and $\gamma$ be irreducible representations of $\gl_q(\IC)$ and $\gl_s(\IC)$, respectively.  Then $\rho\otimes\gamma$ occurs in $\tau_{q, s}^d$ if and only if $\rho$ and $\gamma$ have highest weights
\begin{align*}
r_1\geq\cdots\geq r_s\geq 0\geq \cdots\geq 0 \mbox{ and } r_1\geq \cdots \geq r_s &\hspace{1cm} \mbox { if } q\geq s\\
r_1\geq\cdots\geq r_q \mbox{ and } r_1\geq \cdots \geq r_q\geq 0\geq \cdots\geq 0 & \hspace{1cm} \mbox { if }q\leq s,
\end{align*}
respectively,
with
\begin{align*}
r_1+\cdots +r_\mu = d,
\end{align*}
where $\mu = \min\left\{q, s\right\}$.  Furthermore, such a representation $\rho\otimes \gamma$ occurs in $\tau^d_{q, s}$ with multiplicity one.  The corresponding irreducible subspace of $\mathfrak{S}_d\left(M_{q\times s}(\IC)\right)$ contains a polynomial $\mathcal{P}$ defined by
\begin{align*}
\mathcal{P}(z) = \prod_{j=1}^\mu{\det}_j(z)^{e_j},
\end{align*}
as an eigenvector of highest weight with respect to both $\rho$ and $\gamma$, where $e_j$ is defined as in Equations \eqref{ej1} and \eqref{ej2} for $1\leq j\leq \mu$.
\end{thm}

Now that we have described the representations occurring in $\tau_{r, s}^d$ over $\IC$, we describe the representations occurring in the restriction of $\tau_{n,n}^d$ to $\left(\left(\gl_r(S)\times\gl_s(S)\right)\times \left(\gl_s(S)\times\gl_r(S)\right)\right)$.

\begin{prop}\label{prop127}
Let $r$ and $s$ be positive integers such that $r+s = n$, and let $d$ be a positive integer.  Then the representation $\tau^d_{n,n}|_{\left(\left(\gl_r(S)\times\gl_s(S)\right)\times \left(\gl_s(S)\times\gl_r(S)\right)\right)}$ of $\left(\gl_r(S)\times\gl_s(S)\right)\times \left(\gl_s(S)\times\gl_r(S)\right)$ on $\mathfrak{S}_d\left(M_{n\times n}(S)\right)$ is isomorphic to
\begin{align*}
\bigoplus_{\stackrel{d_1, d_2, d_3, d_4\in\ZZ_{\geq0}}{ d_1+d_2+d_3+d_4 = d}}\tau_{d_1, d_2, d_3, d_4},
\end{align*}
where 
\begin{align}\label{d1234}
\tau_{d_1, d_2, d_3, d_4} = \tau_{r, r}^{d_1}\otimes\tau_{s, s}^{d_2}\otimes\tau_{r, s}^{d_3}\otimes \tau_{s, r}^{d_4}.
\end{align}
\end{prop}
With regard to Equation \eqref{d1234}, our convention is that if $\gamma$ is a representation, then $\gamma^0$ is the trivial representation.
\begin{proof}
Given positive integers $k$ and $l$, denote by $0_{k, l}$ the $k\times l$ $0$-matrix.  Consider the $n\times n$  complex matrices
\begin{align*}
\alpha &= \begin{pmatrix}
A & 0_{r, s}\\
0_{s, r}& B
\end{pmatrix}\\
\beta & = \begin{pmatrix}
C & 0_{s, r}\\
0_{r, s}& D
\end{pmatrix}\\
z & =  \begin{pmatrix}
X & Y\\
Z& W
\end{pmatrix},
\end{align*}
with $A, D\in \gl_r(S)$, $B, C\in \gl_s(S)$, $X\in M_{r\times s}(S)$, $Y\in M_{r\times r}(S)$, $Z\in M_{s\times s}(S)$, and $W\in M_{s\times r}(S)$.  

Let $g$ be an element of $\mathfrak{S}_d(M_{n\times n}(S))$.  Observe that
\begin{align*}
\tau^d(\alpha, \beta)g(z)  = g\left(\begin{pmatrix}AX{ }^tC & A Y { }^tD\\
BZ{ }^tC & BW{ }^tD\end{pmatrix}\right).
\end{align*}
The proposition now follows immediately.
\end{proof}

Together, Theorem \ref{thm127} and Proposition \ref{prop127} describe the irreducible subrepresentations of $\tau^d_{n,n}|_{\left(\left(\gl_r(\IC)\times\gl_s(\IC)\right)\times \left(\gl_s(\IC)\times\gl_r(\IC)\right)\right)}$.

\begin{defi}\label{defiamk}
Let $m = \left(m\left(\sigma\right)\right)_{\sigma\in\Sigma}, l = \left(l\left(\sigma\right)\right)_{\sigma\in\Sigma}\in\ZZ_{m\geq 0}$.  Let $\rho$ be as in Equation \eqref{autfactorrho}.  Let $\gamma$ be the subrepresentation of $\otimes_{\sigma\in\Sigma}\tau_{m(\sigma), l(\sigma), 0, 0}$ defined by $\gamma = \otimes_{\sigma\in\Sigma}\det^{m(\sigma)}\otimes \det^{l(\sigma)}\otimes \det^0\otimes \det^0$.
We define an operator
\begin{align*}
A_{m, l}^{k, \nu}\left(\ci\right): \mathcal{M}_{k, \nu}\left(\ci, U^{\varphi\oplus-\varphi}\right)\rightarrow \mathcal{M}_{(k+m+l, \nu-l), (k+m+l, \nu-m)}\left(\ci, U^\varphi\times U^{-\varphi}\right)
\end{align*}
by
\begin{align*}
A_{m, l}^{k, \nu}(\ci):=A_{\rho}^\gamma\left(\ci\right).
\end{align*}
\end{defi}

We use the notation
\begin{align*}
A_{m, l}^{k, \nu}(\ci, \sigma) = A_{m(\sigma), l(\sigma)}^{k(\sigma), \nu(\sigma)}(\ci)
\end{align*}
to denote the action of $A_{m, l}^{k, \nu}(\ci)$ on each automorphic form at $\sigma$, for all $\sigma\in\Sigma$.

\begin{prop}
The operators $A_{m, l}^{k, \nu}\left(\ci, \sigma\right)$ exist for all $\sigma\in\Sigma$ and all nonnegative integers $m$ and $l$.
\end{prop}

\begin{proof}
This follows immediately from Proposition \ref{prop127}.
\end{proof}

Note that the definition of the operators $A_{m, l}^{k, \nu}\left(\ci\right)$ depends upon the choice of a Hermitian form $\varphi$.  When it is not clear from context what our choice of Hermitian form is, we shall write
\begin{align*}
A_{m, l}^{k, \nu, \varphi}\left(\ci\right)
\end{align*}
for the operator $A_{m, l}^{k, \nu}\left(\ci\right)$ dependent on the Hermitian form $\varphi$.

Define
\begin{align*}
\iota_0: \mathfrak{Z}^{-\varphi}\times\mathfrak{Z}^{\varphi}\isomto \mathfrak{Z}^{\varphi}\times\mathfrak{Z}^{-\varphi}
\end{align*}
by
\begin{align*}
\iota_0(z, w) = (w, z)
\end{align*}
for all $(z, w)\in\mathfrak{Z}^{-\varphi}\times\mathfrak{Z}^{\varphi}$.  Note the symmetry between the definition of $A_{m, 0}^{k, \nu}\left(\ci, \sigma\right)$ and the definition of $A_{0, l}^{k, \nu}\left(\ci, \sigma\right)$.  As a direct result of this symmetry, we obtain the following lemma
\begin{lem}\label{symherm}
For each Hermitian form $\varphi$, $\ci$-automorphic form $f$ on $U^{\varphi\oplus-\varphi}$, and $m, l\in\ZZ_{\geq 0}^\Sigma$,
\begin{align*}
A_{m, l}^{k, \nu, \varphi}\left(\ci, \sigma\right)f  = \left(A_{l, m}^{k, \nu, -\varphi}\left(\ci, \sigma\right)f\right)\circ\iota_0.
\end{align*}
\end{lem}
\begin{proof}
The lemma follows immediately from the definitions of $A_{m, 0}^{k, \nu}\left(\ci, \sigma\right)$, $A_{0, l}^{k, \nu}\left(\ci, \sigma\right)$, and $\iota_0$.
\end{proof}

Let $A_m^k(\sigma)$ denote the $\sigma$-component of Shimura's operator $A_m^k$ from Expression \eqref{shamk}.

\begin{prop}\label{propshme}
Let $m = \left(m(\sigma)\right)_{\sigma\in\Sigma}\in\ZZ^{\Sigma}_{>0}$.  If $a_\sigma\geq b_\sigma$, the operator $A_{m, 0}^{k, \nu}(\ci, \sigma)$ is the operator $A_m^k(\sigma)$.  If $a_\sigma< b_\sigma$, the operator $A_{0, m}^{k, \nu}(\ci, \sigma)$ is the operator $A_m^k(\sigma)$.
\end{prop}
\begin{proof}
Shimura defines the operators $A_m^k$ in terms of his specific choice of coordinates, whereas our definition is coordinate-free.  That the definitions are equivalent follows from plugging in the choice of coordinates with which Shimura works in \cite{sh}.
\end{proof}
\begin{rmk}
Shimura also handles the more general case of $A_m^k$ in which $m(\sigma)$ is allowed to be negative for some $\sigma\in\Sigma$.  However, that requires antiholomorphic weight-lowering operators, which we do not use here or in any of the previous papers upon which the results of the present paper build.
\end{rmk}

In \cite[Sections 23.8-23.14]{sh}, Shimura describes precisely how the operators $A_m^k$ act on certain $\ci$-Eisenstein series, in terms of his choice of coordinates.  (In the present paper, we do not need the details of that description.)  Shimura's writing is apparently restricted to the operators $A_m^k$, i.e. the operators $A_\rho^\gamma$ with $\gamma$ one-dimensional.

\subsection{More on representations}\label{hwreps}
Like in \cite[Section 4]{apptoSHLvv}, in order to transition from the $\ci$-setting to the $p$-adic setting, we shall briefly discuss rational representations and vector bundles.  Note that we can rephrase the discussion from Sections \ref{highwtirrpoly-section} and \ref{subrepres} in the notation of this discussion.
Our discussion in this section is similar to the discussion in \cite[Section 4.3]{apptoSHLvv}, and as much as possible, we adhere to the notation of \cite[Section 4.3]{apptoSHLvv}.
Our discussion summarizes relevant details from \cite[Section 8.1.2]{hida}, which, in turn, summarizes results from \cite{hidaGME} and \cite{jantzenRAG}.  

For any ring $A$, let $B(A)$ denote the upper triangular Borel subgroup of $\gln\left(\OK\otimes_\ZZ A\right)$, and let $N(A)$ denote the unipotent subgroup of $B(A)$.  Let $T(A)\cong B(A)/N(A)$ denote the torus.  For each character $\kappa$ of $T(A)$, we denote by $R_{\OK}[\kappa](A)$ the $A\left[\gln\left(\OK\otimes_\ZZ A\right)\right]$-module defined by
\begin{multline*}
R_{\OK}[\kappa](A):=\Ind_{B(A)}^{\gln\left(\OK\otimes_\ZZ A\right)}\left(\kappa\right)\\
=\left\{\mbox{homogeneous polynomials } f: \gln\left(\OK\otimes_\ZZ A\right)\rightarrow A\mid\right.\\ \left.
 f(ht) = \kappa(t) f(h) \mbox{ for all } t\in T(A), h\in\gln(A)/N(A)\right\}.
\end{multline*} 
The group $\gln\left(\OK\otimes_\ZZ A\right)/N(A)$ acts on $R_{\OK}[\kappa](A)$ via
\begin{align}\label{hwkappa}
\left(g\cdot f\right)(x) = f\left(g^{-1} x\right).
\end{align}
As in \cite{apptoSHLvv}, we are especially interested in the cases in which $A = \IC$ or $A$ is a $p$-adic ring.  Note that
\begin{align*}
\gln\left(\OK\otimes_\ZZ \IC\right)&\cong \prod_{\sigma\in\Sigma}\gln(\IC)\times\gln(\IC)\\
\gln\left(\OK\otimes_\ZZ \ZZ_p\right)&\cong  \prod_{v\in\Sigma}\gln\left({\Oe}_v\right)\times\gln\left({\Oe}_v\right),
\end{align*}
and for any $p$-adic ring $A$,
\begin{align*}
\gln\left(\OK\otimes_\ZZ A\right)\cong \gln\left(\Oe\otimes A\right)\times \gln\left(\Oe\otimes A\right).
\end{align*}

When $A$ is a topological ring, we define $A\left[\gln\left(\OK\otimes_\ZZ A\right)\right]$-modules
\begin{align*}
\mathcal{C}_{\OK}[\kappa](A):=&\left\{\mbox{continuous functions} f: \gln\left(\OK\otimes_\ZZ A\right)/N(A)\rightarrow A\mid\right.\\ &\left. f(ht) = \kappa(t)f(h) \mbox{ for all } t\in T(A), h\in\gln(A)/N(A)\right\}\\
\mathcal{LC}_{\OK}[\kappa](A):=&\left\{\mbox{locally constant functions} f: \gln\left(\OK\otimes_\ZZ A\right)/N(A)\rightarrow A\mid\right.\\ &\left.
 f(ht) = \kappa(t)f(h) \mbox{ for all } t\in T(A), h\in\gln(A)/N(A)\right\}.
\end{align*}
The group $\gln\left(\OK\otimes_\ZZ A\right)/N(A)$ acts on the elements $f$ of each of these modules via
\begin{align*}
\left(g\cdot f\right)(x) = f\left(g^{-1} x\right).
\end{align*}

We write $\rho_\kappa$ to mean the representation of highest weight $\kappa$, i.e. the representation defined by Equation \eqref{hwkappa}.

\subsection{$p$-adic differential operators}\label{padicdiffops-section}

As discussed in detail in \cite{EDiffOps}, the $\ci$-differential operators
\begin{align}
D_{\rho}^d\left(\ci\right): \mathcal{M}_\rho\left(\ci, U^{\varphi\oplus-\varphi}\right)\rightarrow \mathcal{M}_{\rho\otimes\tau^d}\left(\ci, U^{\varphi\oplus-\varphi}\right),
\end{align}
introduced in Expression \eqref{weightraisingnn} (and used in Sections \ref{cidiffops-section}  and \ref{subrepres} to construct the operators $A_\rho^\gamma\left(\ci\right)$ and $A_{m,l}^{k, \nu}\left(\ci\right)$) have a $p$-adic analogue
\begin{align*}
D_{\rho}^d\left(\padic\right): \mathcal{M}_\rho\left(\padic, U^{\varphi\oplus-\varphi}\right)\rightarrow \mathcal{M}_{\rho\otimes\tau^d}\left(\padic, U^{\varphi\oplus-\varphi}\right),
\end{align*}
where $\mathcal{M}_\rho\left(\padic, U^{\varphi\oplus-\varphi}\right)$.  (Note that the notation in this paper differs from the notation in \cite{EDiffOps}.)  The construction of $D_{\rho}^d\left(\padic\right)$ is similar to the construction of $D_\rho^d\left(\padic\right)$ in Section \ref{cidiffops-section}, except that instead of working over $M_\phi(\ci)$, we work over the Igusa tower, and the Hodge decomposition is replaced by the {\it unit root splitting} $\hdr = \uo\oplus\unitroot$, where $\unitroot$ is a sheaf (over the Igusa tower) introduced in \cite{kad} that plays a role analogous to that of $\overline{\uo\left(\ci\right)}$.

We define operators $A_\rho^\gamma\left(\padic\right)$ analogously to how we defined the operators $A_\rho^\gamma(\ci)$ in Definition \ref{pllbkopdefi}.
\begin{defi}\label{pllbkopdefipadic} Let $d = \left(d(\sigma)\right)_{\sigma\in\Sigma}$ be an element of $\ZZ_{\geq 0}^\Sigma$.  Let $\rho$ be a polynomial representation of $\gln \times\gln$, and let $\gamma$ be a subrepresentation of the restriction of $\rho\otimes \tau^d$ to $\gln\times\gln$.   Fix a projection $\pi_\gamma$ from $\rho$ onto $\gamma$.  We define an operator
\begin{align*}
A_{\rho}^\gamma\left(\padic\right): \mathcal{M}_\rho\left(\padic, U^{\varphi\oplus-\varphi}\right)\rightarrow \mathcal{M}_{\gamma}\left(\padic, U^\varphi\times U^{-\varphi}\right)
\end{align*}
by
\begin{align*}
A_\rho^\gamma(\padic)(f):=\pi_\mathcal{\gamma}\circ\left( D_\rho^d\left(\padic\right) f\right)\circ \iota_{\varphi, p}.
\end{align*}
\end{defi}

We also define operators $A_{m, l}^{k, \nu}\left(\padic\right)$ analogously to how we defined the operators $A_{m, l}^{k, \nu}\left(\ci\right)$.
\begin{defi}\label{defiamkpadic}
Let $m = \left(m\left(\sigma\right)\right)_{\sigma\in\Sigma}, l = \left(l\left(\sigma\right)\right)_{\sigma\in\Sigma}\in\ZZ_{m\geq 0}$.  Let $\rho = \det^{k+\nu}\otimes\det^{-\nu}$ be a representation of $\gln\times\gln$.  Let $\gamma$ be the subrepresentation of $\otimes_{\sigma\in\Sigma}\tau_{m(\sigma), l(\sigma), 0, 0}$ defined by $\gamma = \otimes_{\sigma\in\Sigma}\det^{m(\sigma)}\otimes \det^{l(\sigma)}\otimes \det^0\otimes \det^0$.
We define an operator
\begin{align*}
A_{m, l}^{k, \nu}\left(\padic\right): \mathcal{M}_{k, \nu}\left(\padic, U^{\varphi\oplus-\varphi}\right)\rightarrow \mathcal{M}_{(k+m+l, \nu-l), (k+m+l, \nu-m)}\left(\padic, U^\varphi\times U^{-\varphi}\right)
\end{align*}
by
\begin{align*}
A_{m, l}^{k, \nu}(\padic):=A_{\rho}^\gamma\left(\padic\right).
\end{align*}
\end{defi}

We use the notation
\begin{align*}
A_{m, l}^{k, \nu}(\padic, \sigma) = A_{m(\sigma), l(\sigma)}^{k(\sigma), \nu(\sigma)}(\padic)
\end{align*}
to denote the action of $A_{m, l}^{k, \nu}(\padic)$ on each automorphic form at $\sigma$, for all $\sigma\in\Sigma$.

\begin{prop}
The operators $A_{m, l}^{k, \nu}(\padic, \sigma)$ exist for all $\sigma\in\Sigma$ and for all nonnegative integers $m$ and $l$.
\end{prop}

\begin{proof}
This follows immediately from Proposition \ref{prop127}.
\end{proof}

When it is not clear from context what our choice of Hermitian form is, we shall write
\begin{align*}
A_{m, l}^{k, \nu, \varphi}\left(\padic\right)
\end{align*}
for the operator $A_{m, l}^{k, \nu}\left(\padic\right)$ dependent on the Hermitian form $\varphi$.

\begin{rmk}\label{otherrmk}
When $\gamma_{\kappa, \sigma}$ is an irreducible representation a subrepresentation of $\tau_{d_1, 0, 0, 0}$ (for some nonnegative integer $d_1$) of highest weight $\kappa$, we define $A_\rho^\kappa(\padic, \sigma) := A_\rho^{\gamma_{\kappa, \sigma}}(\padic, \sigma)$.  If $\gamma = \otimes_{\sigma\in\Sigma}\gamma_{\kappa, \sigma}$, we define $A_{\rho}^\kappa = \otimes_{\sigma\in\Sigma}A_{\rho}^{\kappa}(\padic, \sigma)$.  If $\rho$ is as in Equation \eqref{autfactorrho}, we define $A_{\kappa}^{k, \nu}\left(\padic\right) := A_\rho^{\kappa}\left(\padic\right).$  Note that when $\gamma_\kappa$ is the $m$-th power of the determinant and $a_\sigma\geq b_\sigma$ for all $\sigma\in\Sigma$, $A_{\kappa}^{k, l}\left(\padic\right) = A_{m, 0}^{k, l}\left(\padic\right)$.  So the operators $A_{\kappa}^{k, l}\left(\padic\right)$ are a natural generalization of the operators $A_{m, 0}^{k, l}\left(\padic\right)$.  We can clearly define analogous operators by considering irreducible sub-representations of $\tau_{0, d, 0, 0}$, $\tau_{0, 0, d, 0}$, or $\tau_{0, 0, 0, d}$ instead of $\tau_{d, 0, 0, 0}$.  Since the idea is the same for any of these representations and since the ideas are the same in our use of these operators below, we will just assign explicit notation to the operators $A_\kappa^{k, l}$.
\end{rmk}

Given an automorphic form $f$ defined over a subring $R$ of $\bar{\IQ}\cap\iota_p^{-1}\left(\OCp\right)$, we may view $f$ as a holomorphic (and, hence, $\ci$) automorphic form after the extension of scalars $\iota_\infty: R\hookrightarrow \IC$.  (Similarly, after the extension of scalars $\iota_p: R\hookrightarrow\OCp$, we may view $f$ as an element of the space of false $p$-adic automorphic forms introduced at the end of Section \ref{autoforms-section}.)  As explained in \cite{EDiffOps} and used in a crucial way in \cite{apptoSHL, apptoSHLvv}, for any automorphic form $f$ defined over a subring $R$ of $\bar{\IQ}\cap\iota_p^{-1}\left(\OCp\right)$, the values of $D_{\rho}^\gamma\left(\padic\right)f$ and $D_{\rho}^\gamma\left(\ci\right)f$ agree at CM points, up to a period.  The CM points are parametrized by the points of $G^\varphi\times G^{-\varphi}\subseteq G^{\varphi\oplus-\varphi}$ with $\varphi$ {\it definite}; so when $\varphi$ is definite, $A_\rho^{\gamma, \varphi}\left(\ci\right)f$ agrees with  $A_{\rho}^{\gamma, \varphi}\left(\padic\right)f$ up to a period.  If $\varphi$ is not definite, though, not all the points of $G^\varphi\times G^{-\varphi}\subseteq G^{\varphi\oplus-\varphi}$ parametrize CM points; in this case, we do not have a way to compare $A_{\rho}^{\gamma, \varphi}\left(\padic\right)f$ and $A_\rho^{\gamma, \varphi}\left(\ci\right)f$ (except at CM points), and we merely know that the first automorphic form is a $\ci$-automorphic form and the latter one is a $p$-adic automorphic form.

\subsubsection{More on $p$-adic automorphic forms and $p$-adic differential operators}

Now, we review some basic facts about the Igusa tower and $p$-adic automorphic forms.  For details, see \cite[Chapter 8]{hida}; as much as possible, we follow the notation of \cite{hida}.  Given the ordinary locus $S$ (i.e. of ordinary abelian varieties, together with additional structure, so $S= S_{\varphi\oplus-\varphi}$, for example) over a $p$-adic ring $R$, we write $S_m$ to denote $S\times_R R/p^mR$.  Recall that for each nonnegative integer $r$, there is a certain \'etale scheme $T_{m, r}$ together with an action of $\gl_n\left(\ZZ/p^r\ZZ\right)$.  The Igusa tower is the formal scheme $T_{\infty, \infty} = \varinjlim_m\varprojlim_rT_{m,r}$, which is an \'etale covering of $S_{\infty}:=\varinjlim_m S_m.$  We put $V_{m,r} :=H^0\left(T_{m,r}, \mathcal{O}_{T_{m,r}}\right)$.  We put $V:=\varprojlim_m\cup_rV_{m,r}$; as explained in \cite[Section 8.1.1]{hida}, the group $\gln\left(\ZZ_p\right)$ acts naturally on $V$.  When we need to be clear about the relevant Hermitian form $\varphi$, we write $V_\varphi$ in place of $V$.  Following \cite[Section 8.1.1]{hida}, we write $V^N$ for the subspace of $V$ fixed by the unipotent subgroup $N$, and we write $V^N[\kappa]$ for the subspace on which the torus $T$ in the Borel subgroup $B$ acts via $\kappa$ (where $\kappa$ denotes a highest weight, as above).  

As explained after \cite[Remark 8.1]{hida}, there is a map $\beta_\kappa$ such that for an irreducible representation $\rho_\kappa$ of highest weight $\kappa$, $\beta_\kappa: \mathcal{M}_{\rho_\kappa}\left(\padic, U^{\varphi\oplus-\varphi}\right)\rightarrow V^N[\kappa]$; note that this induces a map $\beta$ from the space of automorphic forms (viewed as sections over the ordinary locus) into the space of $p$-adic automorphic forms $V^N$.  (In brief, given $f\in\mathcal{M}_{\rho_\kappa}\left(\padic, U^{\varphi\oplus-\varphi}\right)$ and $A\in S$, we can view as an element $f(A)\in R_{\OK}[\kappa](R)$.  Then $\beta(f)$ is defined by $(\beta(f))(A) = \ell_{can}(f(A)):=(f(A))\left(1_n\right)$ for all $A$.)  Let $\incl$ denote the natural map from the space $\mathcal{M}_{\varphi\oplus-\varphi}$ of automorphic forms arising over the whole locus $M_{\varphi\oplus-\varphi}$ (rather than just over the ordinary locus $S$) to the space of automorphic forms over the ordinary locus.  Then, as explained in \cite[Theorem 8.14]{hida}\footnote{Although the statements in \cite{hida} often use $V$ instead of $V^N$, we alert the reader to the fact that Hida notes at the beginning of \cite[Section 8.3.2]{hida} that he writes $V$ to mean $V^N$ throughout Section 8.3.2.}, the composition $\beta\circ \incl$ is an injection from $\mathcal{M}_{\varphi\oplus -\varphi}$ into $V^N$.  Henceforth, whenever practical, we shall write $V$ in place of $V^N$ and $V_\varphi$ in place of $V_\varphi^N$ (following Hida's convention, as explained in the footnote).

We put $\theta_\kappa^{k, \nu}:=\theta_\kappa^{k, \nu}\left(\padic\right) := \beta\circ A_\kappa^{k, \nu}\left(\padic\right)$.  Observe that the image of $\theta_\kappa^{k, \nu}$ is contained in $V_\varphi\otimes V_{-\varphi}$.

\subsubsection{Action of the $p$-adic differential operators on $q$-expansions}\label{qexp-section}

In \cite[Appendix A4]{sh}, Shimura discusses Fourier expansions of automorphic forms on $U^{\varphi\oplus-\varphi}$; he also describes the pullbacks to smaller groups (e.g. $U^\varphi\times U^{-\varphi}$), in which case one obtains a Fourier-Jacobi expansions with theta functions as the coefficients.  For $U^{\varphi\oplus-\varphi}$ (i.e. a unitary group of signature $(n,n)$), there are algebraic $q$-expansions and an algebraic $q$-expansion principle (\cite[Prop 7.1.2.15]{la}) and a $p$-adic $q$-expansion principle over the Igusa tower (\cite[Corollary 10.4]{hi05}, \cite[Section 8.4]{hida}).  (N. B. The algebraic $q$-expansion principle in \cite[Prop 7.1.2.15]{la} only applies to scalar-weight automorphic forms; however, this is sufficient for our purposes.  Indeed, for vector-weight automorphic forms, we only need a $p$-adic $q$-expansion principle, and this is covered by \cite[Corollary 10.4]{hi05} and \cite[Section 8.4]{hida}.)  When $n$ is even and the signature of $\varphi$ is $\left(\frac{n}{2}, \frac{n}{2}\right)$, one can again directly pull back the algebraic $q$-expansions of automorphic forms on $U^{\varphi\oplus-\varphi}$ to $U^\varphi\times U^{-\varphi}$.  For signature $(a, b)$ with $a\neq b$, the situation is more complicated; by combining the results and ideas in the present paper with the recent results in \cite{CEFMV}, though, the author - together with Fintzen, Mantovan, and Varma - will explain in a forthcoming paper how to handle arbitrary signatures.  (Note that the ideas from the present paper are key ingredients even in the forthcoming paper handling the case of arbitrary signature.)

In Section \ref{measure-section}, we will construct $p$-adic measures whose values include certain $p$-adic automorphic forms on $U^\varphi\times U^{-\varphi}$.  Like in \cite{kaCM, apptoSHL, apptoSHLvv}, our methods rely crucially on the $p$-adic $q$-expansion principle, and thus, our discussion in Section \ref{measure-section} is only as extensive as the present state of the literature on $q$-expansion principles allows.

As explained in \cite[Section 8.4]{hida}, to apply the $p$-adic $q$-expansion principle, it is enough to check the cusps parametrized by points of $GM_+\left(\adeles_E\right)$, where $GM_+$ denotes a certain Levi subgroup of $GU_+$.  (More details about cusps appear in \cite[Chapter 8]{hida} and, as a summary, in \cite{apptoSHLvv}; we will not need the details here.)

The precise action of the differential operators $D_{\rho}^d\left(\padic\right)$ is given in \cite[Section IX]{EDiffOps} and applied in \cite{apptoSHL, apptoSHLvv}.  Suppose that $a_\sigma = b_\sigma$ for all $\sigma\in\Sigma$.  So $n = 2a_\sigma = 2b_\sigma$ for all $\sigma\in\Sigma$.  In this case, if $f$ is a $p$-adic automorphic form on $U(n,n)$ of weight $\rho$ with $q$-expansion
\[
f(q) = \sum_{\beta}c(\beta) q^\beta
\]
and we write $\beta = \left(\begin{smallmatrix}A_{11} & { }^t\overline{A_{21}}\\ A_{21}& A_{22}\end{smallmatrix}\right)$ with $A_{11}, A_{22}\in\hern(K)$ and $A_{21}\in M_{n\times n}(K)$,
it follows from the description of the operators $D_\rho^d\left(\padic\right)$ in \cite{EDiffOps} that
\begin{align*}
\left(\theta_{\kappa}^{k, \nu}(\padic)f\right)(q) &= \sum_{\beta =  \left(\begin{smallmatrix}A_{11} & { }^t\overline{A_{21}}\\ A_{21}& A_{22}\end{smallmatrix}\right)}\phi_{\kappa}\left(A_{21}\right)c(\beta)q_{11}^{A_{11}}q_{22}^{A_{22}}\\
& = \sum_{A_{11}, A_{22}}\sum_{A_{21}}\phi_{\kappa}\left(A_{21}\right)c(\beta)q_{11}^{A_{11}}q_{22}^{A_{22}},
\end{align*}
with $\phi_\kappa$ a polynomial dependent on $\kappa$ (namely a highest weight vector).  Note that we shall also write $\phi_\kappa\left( \left(\begin{smallmatrix}A_{11} & { }^t\overline{A_{21}}\\ A_{21}& A_{22}\end{smallmatrix}\right)\right)$ to mean $\phi_\kappa\left(A_{21}\right)$.  When $\kappa$ is the (highest weight for the) $d$-th power of the determinant, $\phi_\kappa = \det^d$.  As noted in Remark \ref{otherrmk}, we can consider other natural analogues of the operators $A_\kappa^{k, \nu}\left(\padic\right)$ (and hence, of the operators $\theta_\kappa^{k, \nu}\left(\padic\right)$); in these cases, we can similarly describe the $q$-expansions.

\section{$p$-adic families of automorphic forms}\label{measure-section}

Throughout this section, we restrict our discussion to the case in which the Hermitian form $\varphi$ such that at all places $\sigma\in\Sigma$ such that $\varphi_\sigma$ is not definite, the signature of $\varphi_\sigma$ is $\left(n/2, n/2\right)$.  So if $\varphi_\sigma$ is not definite at some $\sigma$, then we must also assume that $n$ is even.  (By combining the results and ideas in the present paper with the recent results in \cite{CEFMV}, the author - together with Fintzen, Mantovan, and Varma - will explain in a forthcoming paper how to extend this discussion to arbitrary signatures.  Note that the ideas from the present paper are key ingredients in that paper.)   Also, throughout this section, let $N\subseteq \gln$ be as in Section \ref{hwreps}, and let
\begin{align*}
N_\varphi &:= \prod_{v\in\Sigma}N_{\varphi, v}\\
N_{\varphi, v}&:= \begin{cases}
N\left({\Oe}_v\right)&\mbox{ if } a_v b_v = 0\\
N_{n/2}\left({\Oe}_v\right)\times N_{n/2}\left({\Oe}_v\right) = \left\{\begin{pmatrix}A &0\\ 0 & B\end{pmatrix}\in N\left({\Oe}_v\right)\right\}&\mbox{ otherwise.}
\end{cases}
\end{align*}

In this section, we construct $V_\varphi\otimes V_{-\varphi}$-valued measures.  (Recall that by $V_\varphi\otimes V_{-\varphi}$, we mean $\left(V_\varphi\otimes V_{-\varphi}\right)^{N_\varphi}$)  In brief, for any (profinite) $p$-adic ring $R$, an $R$-valued $p$-adic measure on a (profinite) compact, totally disconnected topological space $Y$ is a $\ZZ_p$-linear map
\begin{align*}
\mu: \mathcal{C}(Y, \ZZ_p)\rightarrow R,
\end{align*}
or equivalently (as explained in \cite[Section 4.0]{kaCM}), an $R'$-linear map
\begin{align*}
\mu: \mathcal{C}(Y, R')\rightarrow R
\end{align*}
for any $p$-adic ring $R'$ such that $R$ is an $R'$-algebra.  Instead of $\mu(f)$, one typically writes
\begin{align*}
\int_{Y}f d\mu.
\end{align*}
Additional details about $p$-adic measures appear in \cite{wa} and \cite[Section 4.0]{kaCM}.

\subsection{A common starting point}

All the measures $\mu_{\varphi}$ constructed below have a common starting point, namely the Siegel Eisenstein series and the Eisenstein measure constructed in \cite{apptoSHLvv}.

Following the notation of \cite{apptoSHLvv}, for $k\in \ZZ$ and $\nu = \left(\nu(\sigma)\right)_{\sigma\in\Sigma}\in\ZZ^\Sigma$, we denote by $\mathbf{N}_{k, \nu}$ the function
\begin{align*}
\mathbf{N}_{k, \nu}: K^\times&\rightarrow K^\times\\
b&\mapsto \prod_{\sigma\in\Sigma}\sigma(b)^{k+2\nu(\sigma)}\left(\sigma(b)\overline{\sigma}(b)\right)^{-\nu(\sigma)}.
\end{align*}
So for all $b\in\mathcal{O}_E^\times$, $\mathbf{N}_{k, \nu}(b) = \mathbf{N}_{E/\IQ}(b)^k$.

\begin{thm}[Theorem 2 in \cite{apptoSHLvv}]\label{thm2}
Let $R$ be an $\OK$-algebra, let $\nu = \left(\nu(\sigma)\right)\in\ZZ^\Sigma$, and let $k\geq n$ be an integer.    Let
\begin{align*}
F: \left(\OK\otimes\ZZ_p\right)\times M_{n\times n}\left(\Oe\otimes\ZZ_p\right)\rightarrow R
\end{align*}
be a locally constant function supported on $\left(\OK\otimes\ZZ_p\right)^\times\times M_{n\times n}\left(\Oe\otimes\ZZ_p\right)$ that satisfies
\begin{align}\label{equnknualakacm}
F\left(ex, \mathbf{N}_{K/E}(e^{-1})y\right) = \mathbf{N}_{k, \nu}(e)F\left(x, y\right)
\end{align}
for all $e\in \OK^\times$, $x\in \OK\otimes\ZZ_p$, and $y\in M_{n\times n}\left(\Oe \otimes\ZZ_p\right)$.  There is an automorphic form $G_{k, \nu, F}$ (on $U(n,n)$) of weight $(k, \nu)$ defined over $R$ whose $q$-expansion at a cusp $m\in GM_+\left(\adeles_E\right)$ is of the form $\sum_{0<\beta\in L_m}c(\beta)q^\beta$ (where $L_{m}$ is the lattice in $\hern(K)$ determined by $m$), with $c(\beta)$ a finite $\ZZ$-linear combination of terms of the form
\begin{align*}
F\left(a, \mathbf{N}_{K/E}(a)^{-1}\beta\right)\mathbf{N}_{k, \nu}\left(a^{-1}\det\beta\right)\mathbf{N}_{E/\IQ}\left(\det\beta\right)^{-n}
\end{align*}
(where the linear combination is a sum over a finite set of $p$-adic units $a\in K$ dependent upon $\beta$ and the choice of cusp $m$).  When $R = \IC$, these are the Fourier coefficients at $s=\frac{k}{2}$ of the $\ci$-automorphic form $G_{k, \nu, F}\left(z, s\right)$ (which is holomorphic at $s=\frac{k}{2}$) defined in \cite[Lemma 9]{apptoSHLvv}.
\end{thm}
As explained in \cite[Theorem 12]{apptoSHLvv}, which we state below, when $R$ is a $p$-adic ring, we can view $G_{k, \nu, F}$ as a $p$-adic automorphic form (over the ordinary locus).

We define $G_{k, \nu, F, \varphi}$ to be the pullback of the automorphic form $G_{k, \nu, F}$. 

\begin{thm}[Theorem 12 in \cite{apptoSHLvv}]\label{thm12}
Let $R$ be a (profinite) $p$-adic $\OK$-algebra.  Fix an integer $k\geq n$, and let $\nu = \left(\nu(\sigma)\right)_{\sigma\in\Sigma}\in\ZZ^\Sigma$.  Let
\begin{align*}
F:\left(\OK\otimes\ZZ_p\right)\times M_{n\times n}\left(\Oe\otimes\ZZ_p\right)\rightarrow R
\end{align*}
be a continuous function supported on $\left(\OK\otimes\ZZ_p\right)^\times\times\gln\left(\Oe\otimes\ZZ_p\right)$ that satisfies
\begin{align*}
F\left(ex, \mathbf{N}_{K/E}(e)^{-1}y\right) = \mathbf{N}_{k, \nu}(e)F(x, y)
\end{align*}
for all $e\in\OK^\times$, $x\in\OK\otimes\ZZ_p$, and $y\in\gln\left(\Oe\otimes \ZZ_p\right)$.  Then there exists a $p$-adic automorphic form $G_{k, \nu, F}$ in $V_{\varphi\oplus-\varphi}$ whose $q$-expansion at a cusp $m\in GM_+\left(\adeles_E\right)$ is of the form $\sum_{0<\beta\in L_m}c(\beta)q^\beta$ (where $L_m$ is the lattice in $\hern(K)$ determined by $m$), with $c(\beta)$ a finite $\ZZ$-linear combination of terms of the form
\begin{align*}
F\left(a, \mathbf{N}_{K/E}(a)^{-1}\beta\right)\mathbf{N}_{k, \nu}\left(a^{-1}\det\beta\right)\mathbf{N}_{E/\IQ}\left(\det\beta\right)^{-n}
\end{align*}
(where the linear combination is a sum over a finite set of $p$-adic units $a\in K$ dependent upon $\beta$ and the choice of cusp $m$).
\end{thm}

Note that a similar argument to the one in the proof of Theorem \ref{thm12} in \cite[Theorem 12]{apptoSHLvv}) gives the following analogue in $V_\varphi\otimes V_{-\varphi}$.
\begin{lem}\label{thm12forphi}
Let $R$ be a (profinite) $p$-adic $\OK$-algebra.  
Fix an integer $k\geq n$, and let $\nu = \left(\nu(\sigma)\right)_{\sigma\in\Sigma}\in\ZZ^\Sigma$.  Let
\begin{align*}
F:\left(\OK\otimes\ZZ_p\right)\times M_{n\times n}\left(\Oe\otimes\ZZ_p\right)\rightarrow R
\end{align*}
be a continuous function supported on $\left(\OK\otimes\ZZ_p\right)^\times\times\gln\left(\Oe\otimes\ZZ_p\right)$ that satisfies
\begin{align*}
F\left(ex, \mathbf{N}_{K/E}(e)^{-1}y\right) = \mathbf{N}_{k, \nu}(e)F(x, y)
\end{align*}
for all $e\in\OK^\times$, $x\in\OK\otimes\ZZ_p$, and $y\in\gln\left(\Oe\otimes \ZZ_p\right)$.  Then there exists a $p$-adic automorphic form $G_{k, \nu, F, \varphi}$ in $V_\varphi\otimes V_{-\varphi}$ defined by $G_{k, \nu, F, \varphi} = \left\{G_{k, \nu, F\mod p^j, \varphi}\right\}\in\varprojlim_j\left(V_{\varphi}\otimes V_{-\varphi}\right)\otimes\left(R/p^j R\right)$.  When $\varphi$ is non-definite at all places, the $q$-expansion of $G_{k, \nu, F, \varphi}$ is of the form $\sum_{0<\beta = \begin{pmatrix}A_{11} & A_{12}\\ A_{21} & A_{22}\end{pmatrix}\in L_m}c(\beta)q_{11}^{A_{11}}q_{22}^{A_{22}}$ (where $L_m$ is a lattice in $\hern(K)$ determined by a cusp $m\in GM_+\left(\adeles_E\right)$), with $c(\beta)$ a finite $\ZZ$-linear combination of terms of the form
\begin{align*}
F\left(a, \mathbf{N}_{K/E}(a)^{-1}\beta\right)\mathbf{N}_{k, \nu}\left(a^{-1}\det\beta\right)\mathbf{N}_{E/\IQ}\left(\det\beta\right)^{-n}
\end{align*}
(where the linear combination is a sum over a finite set of $p$-adic units $a\in K$ dependent upon $\beta$ and the choice of cusp $m$).
\end{lem}

As explained in \cite[Corollary 13]{apptoSHLvv}, as $p$-adic automorphic forms on $U(n,n)$,
\begin{align*}
G_{k, \nu, F} = G_{n, 0, \mathbf{N}_{k-n, \nu}\left(x^{-1}\mathbf{N}_{K/E}^n(x)\det y\right)F(x,  y)},
\end{align*}
where $F$ is as in Theorem \ref{thm12} and $\mathbf{N}_{k-n, \nu}\left(x^{-1}\mathbf{N}_{K/E}^n(x)\det y\right)F(x,  y)$ denotes the function
\begin{align*}
(x, y)\mapsto \mathbf{N}_{k-n, \nu}\left(x^{-1}\mathbf{N}_{K/E}^n(x)\det y\right)F(x,  y).
\end{align*}
A similar argument shows that, as elements of $V_\varphi\otimes V_{-\varphi}$,
\begin{align*}
G_{k, \nu, F, \varphi} = G_{n, 0, \mathbf{N}_{k-n, \nu}\left(x^{-1}\mathbf{N}_{K/E}^n(x)\det y\right)F(x,  y), \varphi},
\end{align*}
where $F$ is as in Theorem \ref{thm12forphi} and $\mathbf{N}_{k-n, \nu}\left(x^{-1}\mathbf{N}_{K/E}^n(x)\det y\right)F(x,  y)$ denotes the function
\begin{align*}
(x, y)\mapsto \mathbf{N}_{k-n, \nu}\left(x^{-1}\mathbf{N}_{K/E}^n(x)\det y\right)F(x,  y).
\end{align*}

\subsection{$V_\varphi\otimes V_{-\varphi}$-valued measures}

In analogue with \cite[Lemma (4.2.0)]{kaCM} (which handles the case of Hilbert modular forms) and \cite[Lemma 14]{apptoSHLvv}, we have the following lemma
\begin{lem}\label{HFlem}Let $R$ be a $p$-adic $\OK$-algebra.  Then the inverse constructions
\begin{align}
H(x, y) & = \frac{1}{\mathbf{N}_{n, 0}\left(x\mathbf{N}_{K/E}(x)^{-n}\det y\right)}F\left(x, y^{-1}\right)\nonumber\\
F(x, y) & = \frac{1}{\mathbf{N}_{n, 0}\left(x^{-1}\mathbf{N}_{K/E}(x)^{n}\det y\right)}H\left(x, y^{-1}\right)\label{FHrelationshipequ}
\end{align}
give an $R$-linear bijection between the set of continuous $R$-valued functions
\begin{align*}
F: \left(\OK\otimes\ZZ_p\right)^\times\times \gln\left(\Oe\otimes\ZZ_p\right)\rightarrow R
\end{align*}
satisfying
\begin{align}\label{Fcondequ}
F\left(ex, \mathbf{N}_{K/E}(e)^{-1}y\right) = \mathbf{N}_{n, 0}(e)F(x, y)
\end{align}
for all $e\in \OK^\times$ and the set of continuous $R$-valued functions
\begin{align*}
H: \left(\OK\otimes\ZZ_p\right)^\times\times \gln\left(\Oe\otimes\ZZ_p\right)\rightarrow R
\end{align*}
satisfying 
\begin{align*}
H(ex, \mathbf{N}_{K/E}(e)y) = H(x, y)
\end{align*}
for all $e\in \OK^\times$.
\end{lem}
\begin{proof}
The proof follows immediately from the properties of $F$ and $H$.
\end{proof}

As in the notation of \cite[Section 5]{apptoSHLvv}, let 
\begin{align*}
\mathcal{G}_n = \left(\left(\OK\otimes\ZZ_p\right)^\times\times \gln\left(\Oe\otimes\ZZ_p\right)\right)/\overline{\OK^\times},
\end{align*}
where $\overline{\OK^\times}$ denotes the $p$-adic closure of $\OK^\times$ embedded diagonally as $\left(e, \mathbf{N}_{K/E}(e)1_n\right)$ in $\left(\OK\otimes\ZZ_p\right)^\times\times  \gln\left(\Oe\otimes\ZZ_p\right)$.  So Lemma \ref{HFlem} gives a bijection between the $R$-valued continuous functions on $\mathcal{G}_n$ and the $R$-valued continuous functions $F$ on $\left(\OK\otimes\ZZ_p\right)^\times\times\gln\left(\Oe\otimes\ZZ_p\right)$ satisfying Equation \eqref{Fcondequ}.

\begin{thm}\label{measurethm}
There is a $V_\varphi\otimes V_{-\varphi}$-valued $p$-adic measure $\mu_\varphi$ on $\mathcal{G}_n$ such that for each continuous function $H$ on $\mathcal{G}_n$
\begin{align*}
\int_{\mathcal{G}_n}H d\mu_\varphi= G_{n, 0, F, \varphi},
\end{align*}
where $F$ is defined in terms of $H$ as in Equation \eqref{FHrelationshipequ}.
\end{thm}
\begin{proof}
The theorem follows immediately from the $p$-adic $q$-expansion principle.
\end{proof}
\begin{rmk} 
For any function $H$ on $\mathcal{G}_n$, $\int_{\mathcal{G}_n}H du_\varphi$ is the pullback of a $p$-adic automorphic form $G_{n, 0, F}$ in $V_{\varphi\oplus-\varphi}$.  Furthermore, in \cite[Theorem 15]{apptoSHLvv}, we constructed a measure $\mu_{\mathfrak{b}, n}$ such that
\begin{align*}
\int_{\mathcal{G}_n}Hd\mu_{\mathfrak{b}, n} = G_{n, 0 , F}.
\end{align*}
For $H$ on $\mathcal{G}_n$ the pullback of $\mu_{\mathfrak{b}, n}(H)$ is the same as $\mu_\varphi(H)$.
\end{rmk}

We define $G_{k, \nu, F, \varphi, \kappa}:= \theta_{\kappa}^{k, \nu}\left(\padic\right)G_{k, \nu, F}$.  Let $\phi_\kappa$ be a highest weight vector for the weight $\kappa$ (like in Section \ref{qexp-section}).
Note that by the $p$-adic $q$-expansion principle and the action of the operators described in Section \ref{qexp-section},
\begin{align*}
\int_{\mathcal{G}_n}H(x, y)\phi_\kappa\left(\mathbf{N}_{K/E}(x)y^{-1}\right)d\mu = G_{n, 0, F, \varphi, \kappa}
\end{align*}
for each continuous function $H$ on $\mathcal{G}_n$; and furthermore, for any continuous function $F$ supported on $\left(\OK\otimes\ZZ_p^\times\right)^\times\times \gln\left(\Oe\otimes\ZZ_p\right)$ satisfying
\begin{align*}
F\left(ex, \mathbf{N}_{K/E}(e)^{-1}y\right) = \mathbf{N}_{k, \nu}(e)F(x, y)
\end{align*}
for all $e\in\OK^\times$, $x\in\OK\otimes\ZZ_p$, and $y\in\gln\left(\Oe\otimes\ZZ_p\right)$,
and extended by zero to all of $\left(\OK\otimes\ZZ_p^\times\right)\times M_{n\times n}\left(\Oe\otimes\ZZ_p\right)$,
\begin{align}\label{diffopmeasure}
\int_{\mathcal{G}_n}\frac{1}{\mathbf{N}_{k, \nu}\left(x\mathbf{N}_{K/E}(x)^{-n}\det y\right)}F\left(x, y^{-1}\right)\phi_\kappa\left(\mathbf{N}_{K/E}(x)y^{-1}\right)d\mu_\varphi = G_{k, \nu, F, \varphi, \kappa}.
\end{align}

Also, if we considered the other analogous differential operators discussed Remark \ref{otherrmk}, we would obtain other equations analogous to Equation \eqref{diffopmeasure}.  As explained in Remark \ref{otherrmk}, the ideas for these other operators are similar.

\begin{rmk}
Note that not only are the automorphic forms $G_{k, \nu, F, \varphi}$, with $F$ locally constant, pullbacks to $U^\varphi\times U^{-\varphi}$ of automorphic forms on $U^{\varphi\times-\varphi}$, but (as we see from the $q$-expansions and the $p$-adic $q$-expansion principle) the $p$-adic automorphic forms $\theta_\kappa^{k, \nu}\left(\padic\right)G_{k, \nu, F}$ are $p$-adic limits of pullbacks of automorphic forms arising over $U^{\varphi\oplus-\varphi}$.  At first glance, it may seem remarkable that we obtain these $p$-adic automorphic forms not only as limits of automorphic forms over $U^\varphi\times U^{-\varphi}$ but actually as $p$-adic limits of pullbacks of automorphic forms arising over the larger group $U^{\varphi\oplus-\varphi}$.  Note that this should not surprise us too much, though, since as Shimura notes in \cite[Section 23.12]{sh}, application of the analogous operators $A_{m}^{k, \nu, \varphi}\left(\ci\right)$ to automorphic forms similar to ours gives a $\ci$-automorphic form on $U^\varphi\times U^{-\varphi}$ that agrees with the pullback of a $\ci$-Eisenstein series on $U^{\varphi\oplus-\varphi}$ of the type discussed in \cite{Langlands}.  Shimura's work on pullbacks in the context of differential operators in \cite{sh} and elsewhere appears only to discuss the operators in the case of scalar weights; this is why we compare our operators to his only in the scalar situation, although the more general, vector-weight operators are a natural extension of the ones he considers in \cite{sh}.

\end{rmk}

\section{Remarks on other groups and pullbacks}\label{othergroupspullbacks}
We now provide some brief comments on how one can generalize our results from above to other groups and pullbacks.  

\subsection{Comments on other pullbacks}
As above, let $V$ be an $n$-dimensional vector space over $K$ together with a nondegenerate Hermitian pairing $\varphi$.  Let $r$ be a nonnegative integer, and let $V'$ be an $r$-dimensional vector space over $K$ together with a nondegenerate Hermitian pairing $\eta_r$ of signature $(r,r)$.  Let $V'' = V\oplus V'$ be an $n+r$-dimensional vector space over $K$ together with the nondegenerate Hermitian pairing
\begin{align*}
\psi = -\varphi\oplus\eta_r: V'\times V'&\rightarrow K\\
((v_1, v_2), (w_1, w_2))&\mapsto -\varphi(v_1, w_1)+\eta_r(v_2, w_2).
\end{align*}
Then, similarly to the situation in Section \ref{unitarygroupssection}, the identity
\begin{align*}
V'' = V\oplus V'
\end{align*}
induces inclusions of groups
\begin{align}
U^\varphi\times U^\psi&\hookrightarrow U^{\varphi\oplus\psi}\nonumber\\
G\left(U^\varphi\times U^\psi\right)&\hookrightarrow GU^{\varphi\oplus\psi}\label{gpincl}
\end{align}
compatible with an embedding
\begin{align*}
\mathfrak{Z}^\varphi\times\mathfrak{Z}^\psi\hookrightarrow\mathfrak{Z}^{\varphi\oplus\psi},
\end{align*}
with $\mathfrak{Z}^{\varphi\oplus\psi}$ holomorphically isomorphic to the space
\begin{align*}
\h{n+r} = \prod_{\sigma\in\Sigma}\left\{z\in\gl_{n+r}(\IC)\mid i({ }^t\bar{z}-z)>0\right\}.
\end{align*}
The details of one such embedding, in terms of coordinates, are given in \cite[Equation (6.10.2)]{sh}.  Similarly, the inclusion of groups given in \eqref{gpincl} induces a corresponding map of Shimura varieties, analogous to the one discussed in Section \ref{symmspacesshvarieties}.

In this context, like in the special case in which $\psi = -\varphi$ discussed above, we can again discuss pullbacks of automorphic forms.  For a detailed discussion in the $\ci$-case, see \cite[Sections 21-22]{sh}.

Although Shimura seems to have discussed differential operators together with pullbacks only in the context in which $\psi =- \varphi$, our entire construction from above should carry through for the case of more general $\psi$ (i.e. for $\psi =-\varphi\oplus\eta_r$ with $r>0$).

Note that our entire discussion in the earlier sections of the paper carries through to this situation.  We discuss details in this direction in joint work with Wan \cite{EW}.

\subsection{Comments on the case of Siegel modular forms and symplectic groups}\label{symplecticmeasure}

Although our discussion focuses on unitary groups, we briefly explain how our methods extend naturally to the case of Siegel modular forms, i.e. automorphic forms on symplectic groups.  

Note that if we replace the CM field $K$ by the totally real field $E$, then we may replace our Hermitian forms with alternating forms and our unitary groups with symplectic groups.  In particular, for $\eta$ as in Equation \eqref{equ-etadef}, $GU^\eta$ is replaced by $GSp(n)$, and $U^\eta$ is replaced by $Sp(n)$.  We then obtain an inclusion of groups into a symplectic group, analogous to the inclusion \eqref{gpincl}.  This leads to compatible inclusions of symmetric spaces and moduli spaces, analogous to the ones given in Section \ref{symmspacesshvarieties}.

Although the discussion of differential operators and pullbacks in \cite{sh} only deals with the case of unitary groups of signature $(n,n)$, \cite[Section 23]{shar} explains how to handle the symplectic case (i.e. $Sp(n)$) similarly.  Also, \cite[Theorem 12.7]{shar} includes the statement of Theorem \ref{thm127} for the case of symplectic groups.  In \cite[Section 6]{apptoSHLvv}, we outlined how to construct a $p$-adic Eisenstein measure for symplectic groups, in analogue with the construction for unitary groups that gave the $p$-adic Eisenstein series whose pullbacks we used earlier in this paper.  Using the incarnation of these Eisenstein series for symplectic groups, one can then carry through our procedure from above in the context of symplectic groups, thus obtaining a $p$-adic families of Siegel automorphic forms in the setting of symplectic groups.  

It will be interesting to see how the $p$-adic families discussed here relate to the $p$-adic families of Siegel modular forms that Skinner and Urban announced (at the 2010 RTG/FRG workshop on arithmetic geometry at UCLA) they are constructing via pullbacks.

\section{Acknowledgements}

I am grateful to Harris, Skinner, and Urban for helpful conversations and suggestions while I was working on this paper.  I am also very thankful to Hida and K.-W. Lan for answering my questions about $q$-expansions.

\bibliography{eischen}   

\end{document}